\documentclass[11pt,reqno]{amsart}

\usepackage[utf8]{inputenc}
\usepackage[margin=1.25in]{geometry}
\usepackage{hyperref}
\usepackage{appendix}
\usepackage{amsfonts}
\usepackage{amssymb}
\usepackage{stmaryrd} 
\usepackage{amsmath}
\usepackage{amsthm}
\usepackage{mathrsfs}
\usepackage{xcolor}
\theoremstyle{plain}
\newtheorem{theorem}{Theorem}[section]

\newtheorem{lemma}[theorem]{Lemma}
\newtheorem{proposition}[theorem]{Proposition}
\theoremstyle{remark}

\newtheorem{Definition}[theorem]{Definition}

\newtheorem{remark}[theorem]{Remark}
\newtheorem{notation}[theorem]{Notation}

\usepackage{biblatex} 
\numberwithin{equation}{section}
\title[Stochastic wave equation]{Stochastic wave equation with Hölder noise coefficient: well-posedness and small mass limit}

\author{Yi HAN}
\address{Department of Pure Mathematics and Mathematical Statistics, University of Cambridge.
}
\email{yh482@cam.ac.uk}
\thanks{Supported by EPSRC grant EP/W524141/1.}

\begin{document}

\maketitle

\begin{abstract}
    We construct unique martingale solutions to the damped stochastic wave equation
    $$ \mu \frac{\partial^2u}{\partial t^2}(t,x)=\Delta u(t,x)-\frac{\partial u}{\partial t}(t,x)+b(t,x,u(t,x))+\sigma(t,x,u(t,x))\frac{dW_t}{dt},
    $$
    where $\Delta$ is the Laplacian on $[0,1]$ with Dirichlet boundary condition, $W$ is space-time white noise, $\sigma$ is $\frac{3}{4}+\epsilon$ -Hölder continuous in $u$ and uniformly non-degenerate, and $b$ has linear growth. The same construction holds for the stochastic wave equation without damping term. More generally, the construction holds for SPDEs defined on  separable Hilbert spaces with a densely defined operator $A$, and the assumed Hölder regularity on the noise coefficient depends on the eigenvalues of $A$ in a quantitative way. We further show the validity of the Smoluchowski-Kramers approximation: assume $b$ is Hölder continuous in $u$, then as $\mu$ tends to $0$ the solution to the damped stochastic wave equation converges in distribution, on the space of continuous paths, to the solution of the corresponding stochastic heat equation. The latter result is new even in the case of additive noise. 
\end{abstract}

\section{Introduction}

In this paper we are interested in the damped stochastic wave equation with irregularity in both the drift and the diffusion coefficients. On the compact interval $D=[0,1]\subset\mathbb{R}$ the equation can be formulated as 
\begin{equation}\label{1damped}\begin{aligned}
    &\mu \frac{\partial^2 u^\mu}{\partial t^2}=\Delta u^\mu(t,x)-\frac{\partial u^\mu}{\partial t} (t,x)+b(t,x,u^\mu(t,x))+g(t,x,u^\mu(t,x))dW(t,x),\\&
    u^\mu(t,x)=0,\quad x\in \partial D,\\&
    u^\mu(0,x)=u_0(x),\quad \frac{\partial u^\mu}{\partial t}(0,x)=v_0(x),\quad x\in D,\end{aligned}
\end{equation}
with $\mu>0$. The equation models a material with elasticity, experiencing friction, as well as deterministic and random forcing. The $-\frac{\partial u^\mu}{\partial t}$ term models elasticity and $\mu>0$ is the mass density.

We will also consider in parallel the stochastic wave equation without damping, formulated as
\begin{equation}\label{1undamped}\begin{aligned}
    &\mu \frac{\partial^2 u^\mu}{\partial t^2}=\Delta u^\mu(t,x)+b(t,x,u^\mu(t,x))+g(t,x,u^\mu(t,x))dW(t,x),\\&
    u^\mu(t,x)=0,\quad x\in \partial D,\\&
    u^\mu(0,x)=u_0(x),\quad \frac{\partial u^\mu}{\partial t}(0,x)=v_0(x),\quad x\in D.\end{aligned}
\end{equation}

When the coefficients $b$ and $g$ are Lipschitz continuous in $u^\mu$, these SPDEs have been well studied. The well-posedness of a strong solution follows from a fixed point theorem, and as $\mu\to 0$, the damped stochastic wave equation \eqref{1damped} converges to the corresponding stochastic heat equation. Motivated by finite dimensional Langevin systems, it is natural to ask whether we can deduce the same results when the coefficients $b$ and $g$ are not locally Lipschitz continuous. We will give a positive answer in this paper.

In this paper, we will work with equations \eqref{1damped} and \eqref{1undamped} in their abstract forms, formulated as evolution equations on an infinite dimensional, separable Hilbert space $H$. In the literature, \eqref{1damped} and \eqref{1undamped} are usually called random field solutions of SPDEs, which can be regraded as particular examples as abstract evolution equations on Hilbert space. Our abstract treatment not only makes the proof clearer but also shows the result has much wider applicability. We will rewrite \eqref{1damped} and \eqref{1undamped} as the following evolution equation defined on  $H\times H^{-1}:$ (in the setting of \eqref{1damped}, \eqref{1undamped} we take $H=L^2([0,1];\mathbb{R})$, and for a general separable Hilbert space $H$ see Section \ref{sect.2.1} for the definition of $H^{-1}$)
\begin{equation}\label{stract}
\mu \frac{\partial^2 u_\mu(t)}{\partial t^2}=Au_\mu(t) -\zeta \frac{\partial u_\mu(t)}{\partial t} +B(t,u_\mu(t))+G(t,u_\mu(t)) \frac{dW_t}{dt},
\end{equation}
where $W$ is the space-time white noise on $H$ and $A$ is some positive operator whose properties will be specified later. The variable constant $\zeta$ typically takes values 0 or 1. The choice $\zeta=1$ corresponds to the damped stochastic wave equation \eqref{1damped} and the choice $\zeta=0$ corresponds to the wave equation without damping term \eqref{1undamped}. If we take $H=L^2([0,1];\mathbb{R})$, take $A$ to be the Laplace operator on $[0,1]$ with Dirichlet boundary condition, take $W$ the space-time white noise on $L^2([0,1])$ and set $B(t,u_\mu(t))(x)=b(t,x,u^\mu(t,x))$ and $G(t,u_\mu(t))(x)=g(t,x,u^\mu(t,x))$, we recover random field SPDEs \eqref{1damped} and \eqref{1undamped}.

In this paper we are interested in the well-posedness issue of \eqref{stract} when  both $B$ and $G$ are irregular, and study the $\mu\to 0$ limit for the damped nonlinear stochastic wave equation \eqref{1damped} (more precisely, the abstract equation \eqref{stract} with $\zeta=1$). 

Given that $B$ and $G$ are irregular, in this paper it is more convenient to work with probabilistic weak solutions. We will mainly consider \textbf{weak mild} solutions to \eqref{stract}. That is, the solution is weak in probabilistic sense and mild in analytic sense. More precisely: (to simplify notations, we defer the definition of the Hilbert space $\mathcal{H}_1$ and the explicit expression of the mild formulation of $Z_\mu(t)$ to Section \ref{sect.2.1}.)

\begin{Definition}\label{definition1.1.1}
A weak mild solution to the abstract stochastic wave equation \eqref{stract} with initial date $(u_0,v_0)\in\mathcal{H}_1$ is given by a sequence $(\Omega,\mathcal{F},(\mathcal{F}_t),\mathbb{P},W,Z_\mu)$, where $(\Omega,\mathcal{F},(\mathcal{F}_t),\mathbb{P})$ is some filtered probability space defining a cylindrical wiener process $W$ and an $\mathcal{F}_t$-adapted, $H\times H^{-1}$-valued continuous process $(Z_\mu(t))_{t\geq 0}$  (the first coordinate of $Z_\mu(t)$ is $u_\mu(t)$,) such that $\mathbb{P}$-a.s. $Z_\mu(t)$ is a mild solution to \eqref{stract} with initial data $(u_0,v_0)\in\mathcal{H}_1$.  The precise formulation  of mild solution that $Z_\mu(t)$ should satisfy is given in \eqref{eq2.4geg} of Definition \ref{def2.4geg}. 
\end{Definition}

The two main theorems of this paper are as follows. The first considers well-posedness.

\begin{theorem}\label{theorem1.1} Given a separable Hilbert space $H$. Suppose the operator $A$, the nonlinear mapping $B:[0,\infty)\times H\to H$ and $G:[0,\infty)\times H\to \mathcal{L}(H_0,H)$ (with $H_0$ specified below)\footnote{In this paper if the subscript of $|\cdot|$ is not specified, then it refers to the norm of $H$.} satisfy the following list of assumptions:
    
    $\mathbf{H1}$ the operator $A:\mathcal{D}(A)\subset H\to H$ is densely defined, that $A$ has a set of eigenvectors $\{e_n\}_{n\geq 1}$ forming an orthogonal basis of $H$, with eigenvalues $Ae_n=-\alpha_n e_n$. Moreover, we may find some $c>0$ such that $\alpha_n\geq c$ for each $n$, and there exists some $\eta_0\in(0,1)$ such that for each $\eta\in(0,\eta_0)$,
    \begin{equation}\label{intermediatesum}
        \sum_{n\geq 1} \frac{1}{\alpha_n^{1-\eta}}<\infty.
    \end{equation}
Moreover, there exists a Banach space $H_0\subset H$ with a continuous embedding, such that each eigenvector $e_n$ lies in $H_0$, and that for some finite $M<\infty$, 
\begin{equation} \sup_n |e_n|_{H_0}\leq M<\infty.
\end{equation}

    $\mathbf{H2}$ The map $B(t,x):[0,\infty)\times H\to H$ is measurable.
    
    $\mathbf{H3}$ For some $\beta\in(1-\frac{\eta_0}{2},1]$, we may find a constant $M>0$ independent of $t$ such that the function $G(t,\cdot):H\to \mathcal{L}(H_0,H)$ satisfies
    \begin{equation}
        |G(t,x)-G(t,y)|_{\mathcal{L}(H_0,H)}\leq M|x-y|_H^\beta,\quad |x-y|_H\leq 1,\quad x,y\in H.
    \end{equation}
Moreover, $G$ is self-adjoint in the sense that for any $t>0,x\in H$,
\begin{equation}\label{adjoint}
    \langle G(t,x)e_j,e_k\rangle=\langle e_j,G(t,x)e_k\rangle,\quad j,k\in\mathbb{N}_+.
\end{equation}
 $\mathbf{H4}$ We may find an extension of $G(t,x)$ with $G(t,x)\in\mathcal{L}(\widetilde{H}_0,H)$ with $H_0\subset\widetilde{H}_0\subset H$ such that, for each $x\in H$, $G(t,x)$ has a right inverse $G^{-1}(t,x)$ on $H$ in the sense that $$G(t,x)G^{-1}(t,x)u=u \text{ for any }u\in H,$$with $G^{-1}(t,x)H\subset\widetilde{H}_0$ and satisfies, for some $C>0$ independent of $t>0,$
 \begin{equation}
   \sup_{x\in H}|G^{-1}(t,x)|_{\mathcal{L}(H,H)}\leq C<\infty.
 \end{equation}

 $\mathbf{H5}$ the nonlinear mappings $B$ and $G$ have linear growth: for some $M>0,$
 \begin{equation}
     |B(t,x)|_H\leq M(1+|x|_H),\quad |G(t,x)|_{\mathcal{L}(H_0,H)}\leq M(1+|x|_H).
 \end{equation}

    Then there exists a unique weak mild solution to the abstract stochastic wave equation \eqref{stract} with initial data $(u_0,v_0)\in\mathcal{H}_1$, in the sense of Definition \ref{definition1.1.1}.
\end{theorem}

The non-degeneracy assumption $\mathbf{H}4$ cannot be removed if no other additional assumptions are made. Indeed, \cite{gomez2017uniqueness} constructed examples of second order SDEs with diffusion coefficient  $g(t,x)=|x|^\alpha$ for $\alpha\in(0,1)$ where the solutions are nonunique. These counterexamples are excluded by our assumption $\mathbf{H}4$.

\begin{remark}
We give some specific examples where Assumption $\mathbf{H}1$ holds true. The most important case is the Laplacian on $[0,1]$ with Dirichlet boundary condition, then $\mathbf{H}1$ is satisfied with $\eta_0=\frac{1}{2}$ and consequently we can take any $\beta\in(\frac{3}{4},1]$ in Assumption $\mathbf{H}3$. Assumption $\mathbf{H}1$ also covers some random field SPDEs in dimension two or higher, say one considers the bi-Laplacian  $A=-\Delta^2$ on any bounded domain in $\mathbb{R}^d$, $d=1,2,3,$ with Dirichlet boundary conditions.
\end{remark} 

\begin{remark}We show how random field solutions \eqref{1damped} and \eqref{1undamped} to the stochastic wave equation are covered by Assumptions $\mathbf{H}1$ to $\mathbf{H}5$. We take $H=L^2([0,1])$ and $H_0=L^\infty([0,1])$. Since the eigenfunctions of Laplacian on $[0,1]$ are given by $\sin (n\pi x)$, they are uniformly bounded in $L^\infty([0,1])$. This justifies $\mathbf{H}1$. Assume $b$ is $\alpha$-Hölder continuous in $u$, we check $B$ is $\alpha$-Hölder continuous on $H$, justifying assumption $\mathbf{H}2'$ given below: by Cauchy-Schwartz inequality, for any two $u,v\in H$,
$$\begin{aligned}|B(t,u)-B(t,v)|_H=&\sqrt{\int_0^1 |b(t,x,u(x))-b(t,x,v(x))|^2dx}\\\leq &C\sqrt{\int_0^1 |u(x)-v(x)|^{2\alpha}dx}\leq C|u-v|_H^\alpha.\end{aligned}.$$
    The same estimate applies to $g(t,x,u)$ assuming  $g$ is $\beta$-Hölder continuous in $u$. Then notice that, from the definition: for any $u\in H$ and $v\in H_0$, the map $G:[0,\infty)\times H\to \mathcal{L}(H_0,H)$ is given by  $(G(t,u)v)(x)=g(t,x,u(t,x))v(x),x\in[0,1]$, we have,
    \begin{equation}
        |G(t,u)-G(t,v)|_{\mathcal{L}(H_0,H)}\leq C|g(t,x,u(t,x))-g(t,x,v(t,x))|_{L^2([0,1])},
    \end{equation}
    then assumption $\mathbf{H}3$ is justified given the previous estimate. The self-adjointness assumption \eqref{adjoint} follows directly from definition of $G$. For assumption $\mathbf{H}4$, if we assume that $|g(t,x,u)|\geq\epsilon$ for some $\epsilon>0$ and all $t,x,u$, then we may define $G^{-1}:[0,\infty)\times H\to \mathcal{L}(H,H)$ via $G^{-1}(t,u)v(x)=\frac{1}{g(t,x,u(t,x))}v(x),$ where $u,v\in H$. Then it is clear that we can slightly extend the definition of $G$ such that  $GG^{-1}v=v$ for any $v\in H$. We shall assume that $g(t,x,u)$ has linear growth in $u$, so that $|g(t,x,u)|_H\leq C(1+|u|_H)$. This implies $G$ is a well-defined map on $\mathcal{L}(H_o,H)$ and assumption $\mathbf{H}5$ is satisfied.

We note that we can consider some other choices of $H_0$ that lead to new interesting examples of SPDEs with non-local coefficients. Say for example we can take $H_0=H$ and consider $g(t,x,u(x))=1+h(|u(x)-v(x)|_{L^2([0,1])})$ for some $v\in H$ and some $\frac{3}{4}+\epsilon$-Hölder continuous function $h>0$. A lot more different types of SPDEs can be generated under this framework, whose well-posedness are new when the dependence is only Hölder continuous.
    
\end{remark}

The second result considers the $\mu\to 0$ small mass limit of the damped abstract stochastic wave equation \eqref{stract} with $\zeta=1$. The limit is the following $H$-valued stochastic heat equation: 
\begin{equation}\label{shehere!}
    du(t)=[Au(t)+B(t,u(t))]dt+G(t,u(t))dW(t),\quad u(0)=u_0\in H.
\end{equation}
interpreted in terms of mild formulation (see equation \eqref{stochheat}).

The stochastic heat equation \eqref{shehere!} under the regularity assumptions stated in Theorem \ref{2smallnoise} is recently investigated in \cite{han2022exponential}.

\begin{theorem}\label{2smallnoise}
    Assume all the assumptions $\mathbf{H1}$ to $\mathbf{H5}$ are true. Moreover, assume the following stronger assumptions:

    $\mathbf{H}1'$ In addition to $\mathbf{H}_1$, assume we may find some constants $C>0$ such that
    \begin{equation}
\frac{1}{C} k^{\frac{1}{1-\eta_0}}\leq \alpha_k\leq Ck^{\frac{1}{1-\eta_0}} \quad \forall k\in\mathbb{N}_+
\end{equation}
(this assumption is stronger than and trivially implies estimate \eqref{intermediatesum}.   This assumption is satisfied by the Laplacian operator on $[0,1]$ once we take $\eta_0=\frac{1}{2}$).

 $\mathbf{H}2'$: the drift $B$ is Hölder continuous: for some $\alpha\in(0,1]$, we can find a time independent constant $M>0$ such that
 \begin{equation}
     |B(t,x)-B(t,y)|_H\leq M|x-y|_H^\alpha,\quad \forall |x-y|\leq 1,\quad x,y\in H.
 \end{equation} Consider the damped abstract stochastic wave equation, \eqref{stract} with $\zeta=1$.

 Then for any $T>0$, the solution $u_\mu(t)$ (the first component of $Z_\mu(t)\in H\times H^{-1}$) converges to $u(t)$, the solution to the stochastic heat equation \eqref{shehere!}, in terms of convergence in distribution, as $\mathcal{C}([0,T];H)-$valued processes.
\end{theorem}
We will actually prove convergence of $u_\mu|_{[0,T]}$ towards $u|_{[0,T]}$ with respect to a Wasserstein distance on $\mathcal{P}(\mathcal{C}([0,T];H)),$ the space of Borel probability measures on $\mathcal{C}([0,T];H)$.

This paper provides, to our best knowledge, currently the most general results on stochastic wave equation with irregular coefficients. We cover irregular drift and diffusion coefficients, and irregularity in the diffusion coefficient has never been considered before in this context. We work on a separable Hilbert space so that random field solutions are contained as special cases but our theory is much more general. We provide the first proof of Smoluchowski-Kramers approximation for an infinite dimensional system with merely Hölder coefficients. Due to the complexity of the wave equation semigroup (which is not analytic, for example), even the weak existence part of Theorem \ref{theorem1.1} seems novel and is not covered by standard methods (see Section \ref{3.111} for discussions). Before we discuss the method and plan of this paper, we first give a brief literature review of related results.

\subsection{Related literature}
\subsubsection{Works on regularization by noise}
In the broadest scope, this work contributes to investigating the regularization by noise phenomenon for stochastic differential equations with irregular coefficients. The simplest case one may think of is a finite dimensional SDE 
$$dX_t=b(X_t)dt+\sigma(X_t)dW_t$$ defined on $\mathbb{R}^d$.
When $\sigma=I$ is the identity and $b$ is bounded measurable, one can construct a unique weak solution via Girsanov transform, and prove strong uniqueness via Zvonkin's transform \cite{zvonkin1974transformation}. Extensions to unbounded, time dependent coefficients can be found in \cite{krylov2005strong}. Concerning an irregular diffusion coefficient $\sigma$, when $\sigma$ is non-degenerate and continuous, one can always construct a unique weak solution via representing the solution as a time changed Brownian motion (see \cite{stroock1997multidimensional}, Section 6.5), and under some extra Sobolev regularity assumptions on $ \sigma$ one can show the solution is strong. The two claims concerning $\sigma$ 
are available primarily because we work in finite dimensions, so that the solution to the SDE without drift is a time-changed Brownian motion, and the analytical theory of parabolic PDEs can be applied to derive quantitative estimates.

For an infinite dimensional stochastic system, consider the stochastic heat equation
\begin{equation}\label{4she}dX_t=AX_t dt+b(X_t)dt+\sigma(X_t)dW_t\end{equation}
where we may as well take $A$ the Laplacian on $[0,1]$ as a motivating example. There has been much recent investigation of well-posedness with an irregular drift $b$ in the case of additive noise $\sigma=I$: pathwise uniqueness is proved for Hölder continuous drifts $b$ in \cite{da2010pathwise} in the context of evolution equations on Hilbert space, and for bounded measurable drifts \cite{gyongy1993quasi} and distributional drifts \cite{athreya2020well} in the context of random field solutions (but not for evolution equations on infinite dimensional Hilbert space). The investigation of an irregular diffusion coefficient in the infinite dimensional setting turns out to be much more difficult. Notable technical challenges appear in infinite dimensions, as we may no longer represent the solution as a time changed Brownian motion and no analytical PDE theory is available.
In \cite{mytnik2011pathwise} the authors made a great achievement in proving pathwise uniqueness of random field solution to the stochastic heat equation when the drift $b$ is Lipschitz and diffusion coefficient $\sigma$ is $\frac{3}{4}+\epsilon$-Hölder continuous in the function argument. The proof is fairly technical and does not seem to extend to Hilbert space valued abstract SPDEs. Recently in \cite{han2022exponential} the author introduced a direct method to prove uniqueness of martingale solutions of \eqref{4she} with irregular $\sigma$ and $b$ for \eqref{4she} defined on infinite dimensional separable Hilbert space, allowing an additional drift $(-A)^{1/2}F(X_t)$ and studied the long time behavior of solutions.

For the stochastic wave equation \eqref{1damped}, there are a few papers considering irregular drift coefficients and additive white noise, see for example \cite{masiero2017well}, and papers on a related problem of martingale solutions with reflection, see \cite{kim2008martingale}. Very recently, in \cite{bogso2022smoothness} the authors proved pathwise uniqueness and Malliavin differentiability for the random field solution to the hyperbolic SPDE $X_{s,t},(s,t)\in[0,T]^2:$
\begin{equation}\label{hyperbolic}
\begin{aligned}
&\frac{\partial^2 X(s,t)}{\partial s\partial t}=b(s,t,X(s,t))+\frac{\partial^2W(s,t)}{\partial s\partial t},\quad s,t\in[0,T]^2,\\&X(s,0)=X(0,t)=x\end{aligned}
\end{equation}
where $W_{s,t}$ is the Wiener sheet and $b:[0,T]^2\times\mathbb{R}^d\to\mathbb{R}^d$ is bounded measurable. Their method is based on an adaptation of Davie's pathwise argument \cite{davie2007uniqueness} and does not extend to the case of an irregular diffusion coefficient $\sigma$. Moreover, the stochastic wave equation with a damping term $\partial_t u$ is not covered by the results in \cite{bogso2022smoothness}, as they consider the SDE \eqref{hyperbolic} and then recover the wave equation via a $\frac{\pi}{4}$ rotation. Thus Theorem \ref{theorem1.1} in this paper is somewhat orthogonal to \cite{bogso2022smoothness}: we consider (damped) semilinear stochastic wave equation with irregular noise coefficients and solve SPDEs on an infinite dimensional Hilbert space, which has much wider generality and contains the random field SPDEs as special examples. Due to this infinite dimensional feature, establishing pathwise uniqueness is particularly difficult (for the stochastic wave equation with additive noise, pathwise uniqueness of the Hilbert space valued abstract SPDE is currently only established when the drift $B$ is $\beta>\frac{2}{3}$-Hölder continuous, see \cite{masiero2017well}). Meanwhile, weak well-posedness and convergence in distribution are already quite useful in many applications, so we do not discuss the proof of pathwise uniqueness in this paper.

\subsubsection{Works on stochastic wave equation}
This paper is the first, to our best knowledge, to consider stochastic wave equations with non-Lipschitz noise coefficients. Earlier studies on the stochastic wave equation focus on Lipschitz continuous $b$ and $\sigma$ in dimensions 1,2 and 3, working with white noise in dimension 1 and coloured noise in dimension 2 and 3 so that the solutions are function valued. Literature in this direction is extensive so we just cite a few and refer to the references therein: \cite{dalang1998stochastic}, \cite{quer2004stochastic},  \cite{carmona1988random}, \cite{carmona1986introduction}. In recent years, the stochastic wave equation with polynomial nonlinearity, driven by space-time white noise in two and three dimensions have been investigated in a number of papers \cite{gubinelli2018renormalization}, \cite{gubinelli2022global}, \cite{gubinelli2023paracontrolled}. The Da Prato-Debussche trick \cite{da2003strong}  or paracontrolled calculus \cite{gubinelli2015paracontrolled} are used in these works and the solutions are not function valued. Thus in these works one has to consider an additive white noise or sufficiently regular multiplicative noise coefficient.

\subsubsection{Works on Smoluchowski-Kramers approximation}

The second component of this paper, Theorem \ref{2smallnoise}, concerns the small noise limit of stochastic wave equations. This phenomenon was first studied in finite dimensions for the Langevin equation in \cite{kramers1940brownian} and \cite{smoluchowski1916drei} and is then called the Smoluchowski-Kramers approximation. More results on the convergence in finite dimensions can be found in \cite{freidlin2004some}. The first proof of Smoluchowski-Kramers approximation for stochastic wave equation was given in \cite{cerrai2006smoluchowski}, \cite{cerrai2006smoluchowski2}, and later extended in \cite{salins2019smoluchowski} to show that the approximation works for multiplicative noise in any spatial dimension provided that the covariance of the noise is suitably chosen. Some more recent investigations of Smoluchowski-Kramers approximation can be found in \cite{cerrai2006smoluchowski}, \cite{cerrai2016smoluchowski}, \cite{cerrai2022smoluchowski}. In these works it is assumed that the drift and diffusion coefficients of the SPDE are Lipschitz continuous in $u$. In \cite{zine2022smoluchowski} the authors proved Smoluchowski-Kramers approximation for 2D stochastic wave equation with additive space-time white noise and polynomial and sine nonlinearities.

 In a recent work \cite{xie2022smoluchowski} the authors considered a finite dimensional system with Hölder continuous drift and non-degenerate noise, and proved Smoluchowski-Kramers approximation via an application of Zvonkin's transform. In this paper Theorem \ref{2smallnoise} we show that Smoluchowski-Kramers approximation is also valid for irregular coefficients in infinite dimensions, widely extending the scope of the said topic to cases where analytical techniques are not available.

\subsection{Plan of this paper}

The key technique that we will use to treat irregular coefficients is 
a generalized coupling technique introduced in \cite{kulik2020well} in the setting of stochastic delay equations. This technique makes use of the non-degeneracy of noise coefficient $G$ and a combination of both pathwise and probabilistic techniques. Suppose we want to measure the distance between two solutions with same initial condition but different coefficients,
\begin{equation}
\mu \frac{\partial^2 u(t)}{\partial t^2}=Au(t) -\zeta \frac{\partial u(t)}{\partial t} +G(t,u(t)) \frac{dW_t}{dt},
\end{equation}
\begin{equation}
\mu \frac{\partial^2 u^n(t)}{\partial t^2}=Au^n(t) -\zeta \frac{\partial u^n(t)}{\partial t} +G^n(t,u^n(t)) \frac{dW_t}{dt},
\end{equation}
we introduce a control parameter $\lambda$, a stopping time $\tau$ and consider the controlled process
\begin{equation}\label{}
\mu \frac{\partial^2 \widetilde{u}^n(t)}{\partial t^2}=A\widetilde{u}^n(t) -\zeta \frac{\partial \widetilde{u}^n(t)}{\partial t} +\lambda(u(t)-\widetilde{u}^n(t))1_{t\leq\tau}
+G^n(t,\widetilde{u}^n(t)) \frac{dW_t}{dt},
\end{equation} 

In the probabilistic step we estimate the total variation distance between $u^n|_{[0,T]}$ and $\widetilde{u}^n|_{[0,T]}$ via Girsanov transform and relative entropy computations, which is standard. In the analytic step we compare the distance between $u(t)$ and $\widetilde{u}^n(t)$ via pathwise estimates. The pathwise estimate is somewhat more involved in infinite dimensions. In the parabolic literature \cite{han2022exponential} this estimate was achieved via a careful use of the stochastic factorization theorem. In the hyperbolic case of this paper, we further develop the ideas in \cite{han2022exponential}, combining more refined estimates of second order semigroups in \cite{salins2019smoluchowski}. The proof of Theorem \ref{2smallnoise} goes through similar lines, yet a further refined semigroup analysis is needed to make the leading numerical constants independent of the small parameter $\mu$.

This article is organised as follows. In Section \ref{Section2} we develop the necessary functional analytic tools and maximal inequality estimates for stochastic integrals. In section \ref{Section3} we prove existence and uniqueness of a weak-mild solutions, Theorem \ref{theorem1.1}. In Section \ref{Section4} we prove the small mass limit, Theorem \ref{2smallnoise}.

\section{Stochastic estimates of the semigroup}\label{Section2}

We now summarize the functional analytic background of this paper, and give a precise definition to the solutions of the (damped) stochastic wave equation.

\subsection{Functional analytic preliminaries}\label{sect.2.1} For any $\delta\in\mathbb{R}$,
define the interpolation spaces $H^\delta$ (see for example \cite{hairer2009introduction}, Section 4.4) as the completion of the vector space $\operatorname{Vect}(e_1,e_2,\cdots)$ of finite linear sums of $e_k,k\in\mathbb{N}$ with respect to the norm
$$ |f|_{H^\delta}^2=\sum_{k=1}^\infty \alpha_k^\delta \langle f,e_k\rangle_H^2,
$$
and define the phase space $\mathcal{H}_\delta:=H^\delta\times H^{\delta -1}$. We will mostly consider $\mathcal{H}:=\mathcal{H}_0$. 

To make the notations more compact, we keep the undetermined constant $\zeta$ that take two values: 0 and 1. The choice $\zeta=1$ corresponds to the damped stochastic wave equation, and the choice $\zeta=0$ corresponds to the one without damping. We define a linear operator, for both $\zeta=0$ and $\zeta=1$: 
$\mathcal{A}^\lambda
_\mu:D(\mathcal{A}_\mu^\lambda):=\mathcal{H}_{\delta+1}\to \mathcal{H}_\delta$ via
\begin{equation}
    \mathcal{A}_\mu^\lambda(u,v)=(v,(A-\lambda)u/\mu-\zeta v/\mu).
\end{equation}
In the following discussions, whenever $\lambda$ is omitted in the notation, it means that we take $\lambda=0$.

 We give some clarifications for the constants $\lambda,\mu,\zeta$: the constant $\lambda$ is an intermediate parameter and we will set it very large in the proof. It does not appear in the final statements of our theorems. The constant $\mu$ is the parameter that we will tune to zero in the small mass limit. The value of $\mu\neq 0$ is irrelevant when we prove the well-posedness of the wave equation. In the proof of well-posedness, we will consider both cases $\zeta=0,1$, while in the proof of small mass limit only the case $\zeta=1$ matters.

Now we denote by $S_\mu^\lambda(t):\mathcal{H}_\delta\to \mathcal{H}_\delta$ the $C_0$ semigroup generated by $\mathcal{A}_\mu^\lambda$. This semigroup is continuous but not analytic, thus constituting the major analytical difficulty of this paper.

Now define $\Pi_1:\mathcal{H}_\delta\to H^\delta$ to be the projection to the first component and $\Pi_2:\mathcal{H}_\delta\to H^{\delta -1}$ projection to the second component. Now set $\mathcal{I}_\mu:H^\delta\to H_\delta$ via
\begin{equation}
    \mathcal{I}_\mu u=(0,u/\mu).
\end{equation}
Then for each $\lambda\geq 0,$ the solution to the abstract stochastic wave equation with $\zeta=0,1$
\begin{equation}\label{**2dampedabstract}
\mu \frac{\partial^2 u_\mu^\lambda(t)}{\partial t^2}=(A-\lambda)u_\mu^\lambda(t) -\zeta \frac{\partial u_\mu^\lambda(t)}{\partial t} +B(t,u_\mu^\lambda(t))+G(t,u_\mu^\lambda(t)) \frac{dW_t}{dt},
\end{equation}
can be formulated into the following abstract evolution equation. Denote the two coordinates by $Z_\mu^\lambda(t)=(u_\mu^\lambda (t),v_\mu^\lambda(t))$ with initial value $Z_\mu^\lambda(0)=Z_0=(u_0,v_0)\in \mathcal{H}_1$, we can rewrite \eqref{**2dampedabstract} as
\begin{equation}
    dZ_\mu^\lambda(t)=\mathcal{A}_\mu^\lambda  Z_\mu^\lambda(t)dt+\mathcal{I}_\mu B(t,\Pi_1 Z_\mu^\lambda(t))dt+\mathcal{I}_\mu G(t,\Pi_1 Z_\mu^\lambda(t))dW_t.
\end{equation}

The mild solutions to the stochastic wave equations are defined as follows:
\begin{Definition}\label{def2.4geg} The mild solution to the abstract stochastic wave equation \eqref{stract} is defined as the adapted process $Z_\mu$ on $\mathcal{H}_0$ that satisfies, for each $t>0$,
\begin{equation}\label{eq2.4geg}
    \begin{aligned}
Z_\mu(t)&=\mathcal{S}_\mu(t)z_0+\int_0^t \Pi_1 \mathcal{S}_\mu(t-s)\mathcal{I}_\mu B(s,\Pi_1 Z_\mu(s))ds\\&+\int_0^t \Pi_1 \mathcal{S}_\mu (t-s)\mathcal{I}_\mu G(s,\Pi_1 Z_\mu(s))dW_s.\end{aligned}
\end{equation}
More generally, the mild solution to the abstract stochastic wave equation \eqref{**2dampedabstract} is defined as the adapted process $Z_\mu^\lambda$ on $\mathcal{H}_0$ that satisfies, for each $t>0$,
\begin{equation}\label{generalpi1}
    \begin{aligned}
Z_\mu^\lambda(t)&=\mathcal{S}_\mu^\lambda(t)z_0+\int_0^t \Pi_1 \mathcal{S}_\mu^\lambda(t-s) \mathcal{I}_\mu B(s,\Pi_1 Z_\mu^\lambda(s))ds\\&+\int_0^t \Pi_1 \mathcal{S}_\mu^\lambda (t-s)\mathcal{I}_\mu G(s,\Pi_1 Z_\mu^\lambda(s))dW_s.\end{aligned}
\end{equation}

The mild solution to the stochastic heat equation \eqref{shehere!} is defined as the adapted process $u(t)$ on $H$ that satisfies, for each $t>0$,
\begin{equation}\label{stochheat}
    u(t)=\mathcal{S}(t)u_0+\int_0^t \mathcal{S}(t-s)B(s,u(s))ds+\int_0^t\mathcal{S}(t-s)G(s,u(s))dW_s.
\end{equation}

\end{Definition}

Since $A-\lambda $ has eigenvectors $\{e_k\}_{k\geq 1}$ for each $\lambda\geq 0$, the operator $\mathcal{A}_\mu^\lambda$ leaves invariant the linear subspace of $\mathcal{H}$ with the form $\{(u_ke_k,v_ke_k):u_k,v_k\in\mathbb{R}\}$.
Given $u\in H,v\in H^{-1}$, define $u_k=\langle u,e_k\rangle_H$ and $v_k=\langle v,e_k\rangle _{H^{-1}}$, and consider $$f_\mu^\lambda(k)(t,u_k,v_k):=\langle e_k,\Pi_1 \mathcal{S}_\mu^\lambda(t)(u_ke_k,v_ke_k)\rangle _H.$$
Then $f_\mu^\lambda(k)(t,u_k,v_k)$ solves, in the case $\zeta=1$,
\begin{equation}\label{effkmu}
    \mu {f_\mu^\lambda(k)}''+{f_\mu^\lambda(k)}'+(\alpha_k+\lambda) f_\mu^\lambda(k)=0,\quad f_\mu^\lambda(k)(0)=u_k,f_\mu^\lambda(k)'(0)=v_k,
    \end{equation}
and solves, in the case $\zeta=0,$
\begin{equation}\label{effkmuzet=0}
    \mu {f_\mu^\lambda(k)}''+(\alpha_k+\lambda) f_\mu^\lambda(k)=0,\quad f_\mu^\lambda(k)(0)=u_k,f_\mu^\lambda(k)'(0)=v_k.
    \end{equation}

We will use the following estimate from \cite{salins2019smoluchowski}:

\begin{lemma}
Assume $f_\mu^\lambda(k)(t,u,v)$ is a solution to \eqref{effkmu}. Assume $\zeta=1.$ Then 

\begin{itemize}
    \item assume $u=0$ and $1-4\mu (\alpha_k+\lambda)\geq 0$. Then 
\begin{equation}\label{zeta1est1}
    |f_\mu^\lambda(k)(t,0,v)|\leq 4\mu |v|e^{-(\alpha_k+\lambda)t},
\end{equation}and 
\begin{equation}\label{zeta2est2}
    |f_\mu^\lambda(k)'(t,0,v)|\leq 2|v|e^{-(\alpha_k+\lambda) t}.
\end{equation}

\item Assume $u=0$ and $1-4\mu (\alpha_k+\lambda)\leq 0$, then 
\begin{equation}\label{zeta3est3}
    |f_\mu^\lambda(k)(t,0,v)|\leq \frac{\sqrt{4\mu}|v|}{\sqrt{\alpha_k+\lambda}}e^{-\frac{t}{4\mu}}
\end{equation}
and 
\begin{equation}\label{zeta4est4}
    |f_\mu^\lambda(k)'(t,0,v)|\leq 2|v|e^{-\frac{t}{4\mu}},
\end{equation}

\item In general, for any $k\in\mathbb{N}$, $\mu>0$ and arbitrary $u,v\in\mathbb{R}$, 
\begin{equation}\label{zeta5est5}
    \mu |f_\mu^\lambda(k)'(t,u,v)|^2+(\alpha_k+\lambda)|f_\mu^\lambda(k)(t,u,v)|^2\leq \mu|v|^2+(\alpha_k+\lambda)|u|^2.
\end{equation}

\end{itemize}

\end{lemma}

We have the following $\zeta=0$ version of the above estimate:

\begin{lemma}
    Assume $f_\mu^\lambda(k)(t,u,v)$ is a solution to \eqref{effkmuzet=0}. Assume $\zeta=0.$ Then 
    \begin{equation}\label{zeta0est1}
    |f_\mu^\lambda(k)(t,0,v)|\leq \frac{\sqrt{\mu}|v|}{\sqrt{\alpha_k+\lambda}}
\end{equation}
and 
\begin{equation}\label{zeta0est2}
    |f_\mu^\lambda(k)'(t,0,v)|\leq |v|.
\end{equation}
\end{lemma}

\begin{proof}
The solutions to \eqref{effkmuzet=0} are explicitly given by 
\begin{equation}\label{explicit}
    f_\mu^\lambda(k)(t,u,v)=u\cos(\sqrt{\frac{\alpha_k+\lambda}{\mu}}t)+v\sin(\sqrt{\frac{\alpha_k+\lambda}{\mu}}t)\sqrt{\frac{\mu}{\alpha_k+\lambda}}.
\end{equation}

\end{proof}

Now we outline some simple consequences of these estimates.

\begin{lemma}\label{lemma2.3} For any $t\geq 0$, $\mu>0$, $\lambda\geq 0$ and $\zeta=\{0,1\}$, we have that 
\begin{equation}
    \|\Pi_1 \mathcal{S}_\mu^\lambda(t)\mathcal{I}_\mu\|_{\mathcal{L}(H)}\leq 4,\quad \zeta=1,
\end{equation}
and \begin{equation}
    \|\Pi_1 \mathcal{S}_\mu^\lambda(t)\mathcal{I}_\mu\|_{\mathcal{L}(H)}\leq \frac{1}{\sqrt{\alpha_1\mu}},\quad \zeta=0,
\end{equation}

\end{lemma}

\begin{proof}
    The case $\zeta=1$ and $\lambda=0$ is proved in \cite{salins2019smoluchowski}, Lemma 5.3. By definition, $$\Pi_1 \mathcal{S}_\mu^\lambda(t)\mathcal{I}_\mu e_k=f_\mu^\lambda(k)(t,0,1/\mu)e_k.$$ As $e_k$ forms an orthogonal basis of $H$,
    $$\|\Pi_1\mathcal{S}_\mu^\lambda(t)\mathcal{I}_\mu\|_{\mathcal{L}(H)}\leq \sup_{k\in\mathbb{N}}|f_\mu^\lambda(k)(t,0,1/\mu)|.$$ Assume $\zeta=1$, for $k$ with $1-4\mu(\alpha_k+\lambda)>0$ use \eqref{zeta1est1} , and for the case $1-4\mu(\alpha_k+\lambda)\leq 0$ use \eqref{zeta3est3}. Assume $\zeta=0$, use the estimate \eqref{zeta0est1} and the fact that $\alpha_k\geq \alpha_1>0$ for each $k$.
    
\end{proof}

\begin{lemma}\label{lemma 1.4} For any $\mu>0$, $t\geq 0$, $\zeta=\{0,1\}$ and $\lambda\geq 0$,
\begin{equation}
    \|\Pi_1 \mathcal{S}_\mu^\lambda(t)\begin{pmatrix} I\\0\end{pmatrix}\|_{\mathcal{L}(H)}\leq 1.\quad \zeta=0,1.
\end{equation}

\end{lemma}

\begin{proof}
    In both cases $\zeta=0,1$, we have
    \begin{equation}
    \|\Pi_1 \mathcal{S}_\mu^\lambda(t)\begin{pmatrix} I\\0\end{pmatrix}\|_{\mathcal{L}(H)}=\sup_{k\in\mathbb{N}}|f_\mu^\lambda(k)(t,1,0)|.
\end{equation}
In the case $\zeta=1$ this term is less than or equal to 1 thanks to \eqref{zeta5est5}. In the case $\zeta=0$ this is less than or equal to 1 thanks to \eqref{explicit}.
\end{proof}

\begin{lemma}\label{lemma1.5!}  In the case $\zeta=1$,
    consider $N_\mu^\lambda=\max\{k\in\mathbb{N}:1-4\mu(\alpha_k+\lambda)\geq 0\}$, and denote by $P_{N_\mu^\lambda}$ the orthogonal projection onto the subspace spanned by $e_1,\cdots,e_{N_\mu^\lambda}$. Then for any $t>0$,
    \begin{equation}\label{lemma2.51}
        \|\Pi_1 \mathcal{S}_\mu^\lambda(t)\begin{pmatrix} 0\\P_{N_\mu^\lambda}\end{pmatrix}\|_{\mathcal{L}(H)}\leq 4\mu,\quad \zeta=1
    \end{equation}
and 
     \begin{equation}\label{lemma2.52}
        \|\Pi_1 \mathcal{S}_\mu^\lambda(t)\begin{pmatrix} 0\\I-P_{N_\mu^\lambda}\end{pmatrix}\|_{\mathcal{L}(H^{-1},H)}\leq \sqrt{4\mu},\quad \zeta=1.
    \end{equation}
We also have the following $\zeta=0$ version: for any $t\geq 0,\lambda>0$ and $\mu>0$,
\begin{equation}
    \|\Pi_1 \mathcal{S}_\mu^\lambda(t)\begin{pmatrix} 0\\I\end{pmatrix}\|_{\mathcal{L}(H^{-1},H)}\leq \sqrt{\mu},\quad\zeta=0.
\end{equation}
\end{lemma}
    
\begin{proof}
    The first two claims are proved in Lemma 5.5 and 5.6 of \cite{salins2019smoluchowski} when $\lambda=0$. The first expression is bounded by 
    $\sup_{k\leq N_\mu^\lambda}|f_\mu^\lambda(k)(t,0,1)|\leq 4\mu$, via applying \eqref{zeta1est1}. The second expression is bounded by $$\sup_{k> N_\mu^\lambda }\sqrt{\alpha_k}|f_\mu^\lambda(k)(t,0,1)|\leq \sqrt{4\mu}$$ via applying \eqref{zeta3est3}. The last expression (with $\zeta=0$) is bounded by $$\sup_{k\geq 1}\sqrt{\alpha_k}|f_\mu^\lambda(k)(t,0,1)|\leq \sqrt{\mu}$$ via applying \eqref{zeta0est1}.

\end{proof}

\begin{lemma}\label{lemma1.6!}
For any $\mu\in(0,1)$, we have
\begin{equation}\label{1.22go}
    \|\mathcal{S}_\mu^\lambda(t)\|_{\mathcal{L}(\mathcal{H})}\leq \mu^{-1/2}\sup_k \sqrt{\frac{\alpha_k+\lambda}{\alpha_k}}, \quad \zeta=0,1.
\end{equation}
\end{lemma}

\begin{proof}
    The result is proved for $\zeta=1$ and $\lambda=0$ in \cite{salins2019smoluchowski}, Lemma 5.7. We provide a complete proof in our more general case.
    
By definition of $\mathcal{H}$, for $(u,v)\in\mathcal{H}$ and any $t\geq 0$,

$$
 \mu |\mathcal{S}_\mu^\lambda(t)(u,v)\|_\mathcal{H}^2\leq \mu \|\Pi_2 \mathcal{S}_\mu^\lambda(t)(u,v)|_{H^{-1}}^2+|\Pi_1 \mathcal{S}_\mu^\lambda(t)(u,v)|_H^2.$$
From the Fourier decomposition it suffices to bound 
$$\sum_{k=1}^\infty \left(\frac{\mu}{\alpha_k}|f_\mu^\lambda(k)'(t,u_k,v_k)|^2+|f_\mu^\lambda(k)(t,u_k,v_k)|^2
\right).
 $$
 First consider the case $\zeta=0$. From \eqref{explicit}, we see that it is bounded by 
 
 $$\sup_k \frac{\alpha_k+\lambda}{\alpha_k}(\sum_{k=1}^\infty \frac{\mu}{\alpha_k}|v_k|^2+|u_k|^2)\leq \sup_k\frac{\alpha_k+\lambda}{\alpha_k} (u,v)_\mathcal{H}.$$
 
 Then consider the case $\zeta=1$. It suffices to use \eqref{zeta5est5} and we get the same estimate as above.
\end{proof}

\subsection{Estimates in shifted norm}
As will be made clear afterwards, for general $\lambda>0$ it is more convenient to re-define the norm of the interpolation space $H^\delta$. Now we define, for each fixed $\lambda>0$, $H^\delta(\lambda)$ as the completion of the vector space $\operatorname{Vect}(e_1,\cdots,e_k,\cdots)$ with respect to the norm 
$$ |f|_{H^\delta(\lambda)}^2:=\sum_{k=1}^\infty (\alpha_k+\lambda)^\delta \langle f,e_k\rangle_H^2.
$$ This new norm of $H^\delta(\lambda)$ is equivalent to the norm of $H^\delta$ for any $\lambda>0$, but it will be more convenient to use.

We now consider the phase space $$\mathcal{H}(\lambda):= H^0(\lambda)\times H^{-1}(\lambda)=H\times H^{-1}(\lambda).$$
Note that $H=H^0(\lambda)$ so the norm in the first component of $\mathcal{H}$ is the same as that of $\mathcal{H}(\lambda)$. Now we restate many of our previous estimates in the new norms $H^{-1}(\lambda)$ and $\mathcal{H}_\lambda$:

\begin{lemma}\label{lemma1.7!}
    With the same notation as Lemma \ref{lemma1.5!} and \ref{lemma1.6!}, we have
    \begin{equation}
        \|\Pi_1 \mathcal{S}_\mu^\lambda(t)\begin{pmatrix} 0\\I-P_{N_\mu^\lambda}\end{pmatrix}\|_{\mathcal{L}(H^{-1}(\lambda),H)}\leq \sqrt{4\mu},\quad \zeta=1.
    \end{equation}
and
\begin{equation}
    \|\Pi_1 \mathcal{S}_\mu^\lambda(t)\begin{pmatrix} 0\\I\end{pmatrix}\|_{\mathcal{L}(H^{-1}(\lambda),H)}\leq \sqrt{\mu},\quad\zeta=0.
\end{equation}
Moreover, for any $\mu\in(0,1),$ we have
\begin{equation}
    \|\mathcal{S}_\mu^\lambda(t)\|_{\mathcal{L}(\mathcal{H}(\lambda))}\leq \mu^{-1/2}, \quad \zeta=0,1.
\end{equation}
\end{lemma}

\begin{proof}
    The proof is exactly the same. Thanks to \eqref{zeta3est3} and \eqref{zeta0est1}, we have $$\sup_{k> N_\mu^\lambda }\sqrt{\alpha_k+\lambda}\left|f_\mu^\lambda(k)(t,0,1)\right|\leq \sqrt{4\mu},\quad\zeta=1$$ and
    $$\sup_{k\geq 1}\sqrt{\alpha_k+\lambda}\left|f_\mu^\lambda(k)(t,0,1)\right|\leq \sqrt{\mu},\quad\zeta=0.$$ The last claim follows again from \eqref{zeta5est5} when $\zeta=1$ and \eqref{explicit} when $\zeta=0,$ while the fact that we have $\alpha_k+\lambda$ instead of $\alpha_k$ in the definition of the norm of $\mathcal{H}(\lambda)$ allows us to get rid of the $\sup_k\sqrt{\frac{\alpha_k+\lambda}{\alpha_k}}$ factor in \eqref{1.22go}.
\end{proof}

\subsection{Maximal inequalities of stochastic integrals}
In this paper we will extensively use maximal inequalities of stochastic integrals of the following form 
\begin{equation}\label{Gammamulambda}
\Gamma^{\mu,\lambda}(t)=
\int_0^t \mathcal{S}_\mu^\lambda(t-s)\mathcal{I}_\mu \Phi(s)dW_s
\end{equation}
for some adapted, self adjoint process $\Phi\in L^p(\Omega;L^\infty([0,T];\mathcal{L}(H_0,H)))$. When we apply the estimates related to $\Gamma^{\mu,\lambda}$, we will often take $\Phi(s)=G(s,u(s))$ or  $\Phi(s)=G(s,u(s))-G^n(s,u^n(s))$ for some approximations $G^n$ of $G$. By assumption \eqref{adjoint} and our approximating procedure, these choices will make $\Phi(s)$ a self-adjoint adapted processes. In the rest of this paper we will assume self adjointness of $\Phi$ without mentioning.

We will use the factorization formula (\cite{da2014stochastic},Section 5.3.1) for stochastic convolutions which gives, for $0<\alpha<1$ (which we will later choose to be sufficiently small),
\begin{equation}\label{uppergamma}
    \Gamma^{\mu,\lambda}(t)=\frac{\sin(\alpha\pi)}{\pi}\int_0^t (t-s)^{\alpha-1}\mathcal{S}_\mu^\lambda(t-s)\Gamma_\alpha^{\mu,\lambda}(s)ds,
\end{equation}
where
\begin{equation}\Gamma_\alpha^{\mu,\lambda}(t)=\int_0^t (t-s)^{-\alpha}\mathcal{S}_\mu^\lambda(t-s)\mathcal{I}_\mu\Phi(s)dW_s.
\end{equation}

We first prove maximal inequality estimates for $\Gamma_\alpha^{\mu,\lambda}$. The proof is similar to \cite{salins2019smoluchowski}, Lemma 6.1 but we make a better use of the summability of $\sum_n\frac{1}{\lambda_n^\sigma}$, $\sigma<1$, and our result can be re-interpreted as a maximal inequality of Ornstein-Uhlenbeck process when $\lambda$ is very large. 

\begin{notation}
We abbreviate, for any $T>0$,$$\|\Phi\|_T:=\sup_{0\leq t\leq T}|\Phi(t)|_{\mathcal{L}(H_0,H)}.$$
\end{notation}

\begin{proposition}\label{proposition 1.8} Fix $T>0$.
Assume that there is some $\eta_0\in(0,1)$ such that \begin{equation}\label{summoperator}
\sum_n \frac{1}{\alpha_n^{1-\eta}}<\infty
\end{equation}
for each $\eta\in(0,\eta_0)$. Then for each $p\geq 2$ and each $\eta\in(0,\eta_0),$ for both $\zeta=0$ and $\zeta=1,$
\begin{equation}
    \mathbb{E}[|\Pi_1\Gamma_\alpha^{\mu,\lambda}(t)\|_H^p]\leq C_{\eta,p,\mu}\|\Phi\|_t^p \lambda^{-\frac{\eta p}{2}},\quad t\in[0,T].
\end{equation} The constant $c_{\eta,p,\mu}<\infty$ depends on $\eta$, $p$, $\mu$ and $T$.

\end{proposition}

\begin{proof}
    Applying the  Burkholder-Davis-Gundy inequality to $\Gamma_\alpha^{\mu,\lambda}$:
\begin{equation}
    \mathbb{E}[\|\Pi_1\Gamma_\alpha^{\mu,\lambda}(t)\|_H^p]\leq C\mathbb{E}\left( \sum_{j=1}^\infty\int_0^t (t-s)^{-2\alpha}|\Pi_1\mathcal{S}_\mu^\lambda(t-s)\mathcal{I}_\mu\Phi(s)e_j|_H^2ds\right)^{p/2}.
\end{equation}
 The quadratic variation
$$\begin{aligned}
\Lambda_\alpha^{\mu,\lambda}(t):&=\sum_{j=1}^\infty \int_0^t (t-s)^{-2\alpha}|\Pi_1\mathcal{S}_\mu^\lambda(t-s)\mathcal{I}_\mu \Phi(s)e_j|_H^2ds\\
&=\sum_{k=1}^\infty \sum_{j=1}^\infty \int_0^t (t-s)^{-2\alpha}\langle\Pi_1\mathcal{S}_\mu^\lambda(t-s)\mathcal{I}_\mu\Phi(s)e_j,e_k\rangle_H^2ds\\
&=\sum_{k=1}^\infty\sum_{j=1}^\infty\int_0^t (t-s)^{-2\alpha}\langle \Phi(s)e_j,\mathcal{I}_\mu^*S_\mu^{\lambda*}(t-s)\Pi_1^*e_k\rangle_H^2ds
\end{aligned}
$$

Since $$\begin{aligned}
\langle\mathcal{I}_\mu^{*}\mathcal{S}_\mu^{\lambda*}(t)\Pi_1^*e_k,e_j\rangle_H&=\langle\Pi_1\mathcal{S}_\mu^\lambda(t)\mathcal{I}_\mu  e_j,e_k\rangle_H\\&=
\begin{cases}f_\mu^\lambda(k)(t,0,1/\mu)\quad j=k\\0\quad j\neq k,\end{cases}
\end{aligned}$$ we further simplify $\Lambda_\alpha^{\mu,\lambda}(t)$ as, abbreviating $f_\mu^\lambda(k)(t):=f_\mu^\lambda(k)(t,0,1/\mu)$,

$$
\Lambda_\alpha^{\mu,\lambda}(t)=\sum_{k=1}^\infty\sum_{j= 1}^\infty\int_0^t (t-s)^{-2\alpha}(f_\mu^\lambda(k)(t-s))^2\langle \Phi(s)e_j,e_k\rangle_H^2ds.
$$
By self-adjointness of $\Phi$ we have, using $e_k\in H_0$ with $|e_k|_{H_0}\leq M$,
$$\sum_{j=1}^\infty \langle \Phi(s)e_j,e_k\rangle_H^2=\sum_{j=1}^\infty \langle \Phi(s)e_k,e_j\rangle_H^2=|\Phi(s)e_k|_H^2\leq M|\Phi(s)|^2_{\mathcal{L}(H_0,H)}.$$

    Thus we bound $\Lambda_\alpha^{\mu,\lambda}(t)$ as follows:
    \begin{equation}\label{Lambdamulambda}
\Lambda_\alpha^{\mu,\lambda}(t)\leq\int_0^t (t-s)^{-2\alpha}\|\Phi(s)\|_t^2\left(\sum_{k=1}^\infty f_\mu^\lambda(k)^2(t-s,0,1/\mu)
\right)ds.
    \end{equation}

In the following we find effective upper bounds for $$\sum_{k=1}^\infty f_\mu^\lambda(k)^2(t,0,1/\mu).$$

For the case $\zeta=1$, thanks to \eqref{zeta1est1}, \eqref{zeta3est3}, this is bounded by 
$$\sum_{k=1}^{N_\mu^\lambda}e^{-2(\alpha_k+\lambda)t}+\sum_{k=N_\mu^\lambda+1}^\infty \frac{e^{-\frac{t}{2\mu}}}{\mu(\alpha_k+\lambda)}:=J_1(t)+J_2(t).$$

For the case $\zeta=0$, thanks to \eqref{zeta0est1}, this is bounded by 
$$\sum_{k=1}^\infty \frac{1}{\mu(\alpha_k+\lambda)}:=J_3.$$

For any $\eta\in(0,1)$, by Young's inequality we may find $c_\eta>0$ such that \begin{equation}\label{young's}
\alpha_k+\lambda\geq c_\eta {\alpha_k}^{1-\eta}\lambda^\eta.\end{equation} Thus for any $\eta\in(0,\eta_0)$, by \eqref{summoperator} we may find $c_\eta>0$ such that 

\begin{equation}\label{youngused}J_3\leq \frac{c_\eta}{\mu} \lambda^{-\eta},\quad J_2(t)\leq \frac{c_\eta}{\mu}\lambda^{-\eta}e^{-\frac{t}{2\mu}}.\end{equation}

In the case $\zeta=0$, the proof is already complete by plugging the bound of $J_3$ into \eqref{Lambdamulambda}.

In the following we focus on the case $\zeta=1$. We first compute the stochastic integral relating to $J_1(t)$. The analysis of $J_1(t)$ is similar to that of the heat semigroup in \cite{han2022exponential}. For any $r<1$ we can find positive constants $c_r$, $c_{r,\eta}$ such that, by using \eqref{young's} in the second inequality,
\begin{equation}\label{J1J1}e^{-2(\alpha_k+\lambda)t}\leq c_r\frac{1}{((\alpha_k+\lambda)t)^r}\leq c_{r,\eta}\frac{1}{\alpha_k^{(1-\eta)r}}\lambda^{-\eta r}\frac{1}{t^r}. \end{equation}
Since \begin{equation}\label{579}\int_0^t s^{-2\alpha-r} ds<\infty\end{equation} by choosing $\alpha>0$ sufficiently small followed by choosing $r<1$ sufficiently close to 1, one can ensure that $\sum_{k=1}^\infty\frac{1}{\alpha_k^{(1-\eta)r}}<\infty$ for $r$ sufficiently close to 1, and then $r\eta$, the exponent of $\lambda^{-1},$ can achieve any value in $(0,\eta_0)$ by our choice. Thus we write $\eta$ in place of $r\eta$ for the given $\eta\in(0,1)$, and derive the estimate \begin{equation}\label{eqref580}
    J_1(t)\leq c_{\eta} \lambda^{-\eta}\frac{1}{t^r}
\end{equation} where $c_\eta$ implicitly depends on $r$ but is independent of $t$.
This produces the desired $\lambda^{\eta}$ factor in the upper bound of $\Lambda_\alpha^{\mu,\lambda}$.
Finally we deal with the component involving $J_2(t)$. It suffices to plug in the bound \eqref{youngused} into \eqref{Lambdamulambda}. This finishes the proof with a constant $c_{\eta,p,\mu}$ that is $\mu$-dependent.

We note the constant dependence on the terminal time $T$: $C_{\eta,p,\mu}$ is bounded by (modulo the power in $p$) $\int_0^t (t-s)^{-2\alpha}(J_1(t-s)+J_2(t-s))ds$. This integral is finite for every $t>0$ thanks to \eqref{youngused}, \eqref{579} and \eqref{eqref580}, and is uniformly bounded for any $t\in[0,T]$.

\end{proof}

\begin{proposition}\label{proposition1.10!}
Fix $T>0$.   For each $p\geq 2$ and each $\eta\in(0,\eta_0),$ we may choose $\alpha>0$ sufficiently small such that,
   for $\zeta=1,$ 
\begin{enumerate}
    \item for $\mu\in(0,1)$ and any $t\in[0,T],$
    \begin{equation}\label{1.101.35}
\mathbb{E}|P_{N_\mu^\lambda}\Pi_2\Gamma_\alpha^{\mu,\lambda}(t)|_H^p\leq \frac{c_{\eta,p,\mu}}{\mu^p}\lambda^{-\frac{\eta p}{2}}\mathbb{E}\|\Phi\|_t^p,
    \end{equation}
    \item for $\mu\in(0,1)$ and any $t\in[0,T],$
    \begin{equation}
        \mathbb{E}|(I-P_{N_\mu^\lambda})\Pi_2\Gamma_\alpha^{\mu,\lambda}(t)|_{H^{-1}(\lambda)}^p\leq \frac{c_{\eta,p,\mu}}{\mu^p}\lambda^{-\frac{\eta p}{2}}\mathbb{E}\|\Phi\|_t^p
    \end{equation}
\end{enumerate}
and for $\zeta=0$, for any $\mu\in(0,1)$ and any $t\in[0,T],$
   \begin{equation}
        \mathbb{E}|(I-P_{N_\mu^\lambda})\Pi_2\Gamma_\alpha^{\mu,\lambda}(t)|_{H^{-1}(\lambda)}^p\leq \frac{c_{\eta,p,\mu}}{\mu^p}\lambda^{-\frac{\eta p}{2}}\mathbb{E}\|\Phi\|_t^p.
    \end{equation}
    Here the constant $c_{\eta,p,\mu}<\infty$ depends on $\eta$, $p$, $\mu>0$ and $T$, but does not depend on $\lambda$. The estimates hold uniformly for any $t\in[0,T]$ with the same $T$-dependent numerical constant $C_{\eta,p,\mu}<\infty$.

Moreover, all these estimates are true if we evaluate the left hand side of each estimate at a stopping time $t\wedge \tau$, where $\tau$ is a stopping time adapted to the filtration $(\mathcal{F}_t)_{t\geq 0}$, and we write on the right hand side of each estimate $\sup_{s\in[0,t\wedge\tau]}\|\Phi(s)\|_H^p.$
    
\end{proposition}

\begin{proof} We first deal with the case $\zeta=1$, where we are required to separately deal with the higher and lower modes. For the lower modes, let $\Lambda_1(t)$ denote the quadratic variation of $P_{N_\mu^\lambda}\Pi_2\Gamma_\alpha^{\mu,\lambda}.$ Then 
$$\begin{aligned}\Lambda_1(t)=&\sum_{j=1}^\infty\int_0^t (t-s)^{-2\alpha}|P_{N_\mu^\lambda}\Pi_2\mathcal{S}_\mu^\lambda(t-s)\mathcal{I}_\mu\Phi(s)e_j\|_H^2ds\\
=&\sum_{k=1}^{N_\mu^\lambda }\sum_{j=1}^\infty\int_0^t(t-s)^{-2\alpha}\langle\Phi(s)e_j,\mathcal{I}_\mu^*\mathcal{S}_\mu^{\lambda *}(t-s)\Pi_2^* e_k\rangle_H^2 ds.\end{aligned}$$

The eigenvalues $e_j$ satisfy $\mathcal{I}_\mu^{*}\mathcal{S}_\mu^{\lambda*}\Pi_2^*e_k=f_\mu^\lambda(k)'(t-s)e_k$, where $f_\mu^\lambda(k)$ solves \eqref{effkmu} with $u_k=0$ and $v_k=1/\mu$. By \eqref{zeta2est2} with $v=\frac{1}{\mu},$ for $1\leq k\leq N_\mu^\lambda,$ we have
$$(f_\mu^\lambda(k)'(t)\leq\frac{2e^{-(\alpha_k+\lambda)t}}{\mu},$$ so that, with the same reasoning as the previous proof,
$$\Lambda_1(t)\leq\frac{C}{\mu^2}\int_0^t (t-s)^{-2\alpha}\|\Phi\|_s^2(\sum_{k=1}^{N_\mu^\lambda} e^{-2(\alpha_k+\lambda)(t-s)})ds.$$
    By the same reasoning as in the proof of Proposition \ref{proposition 1.8}, for any $\eta\in(0,\eta_0)$ we may find $C_\eta>0$ such that 
    $$\Lambda_1(t)\leq\frac{C_\eta}{\mu^2}\|\Phi\|_t^2 \lambda^{-\eta}.$$
    By the BDG inequality 
$$ \mathbb{E}|P_{N_\mu^\lambda}\Pi_2\Gamma_\alpha^{\mu,\lambda}(t)|_H^p\leq C\mathbb{E}(\Lambda_1(t))^{p/2}
 $$   which implies \eqref{1.101.35} .

For the higher modes, compute the quadratic variation
$$\begin{aligned}\Lambda_2(t)&=\sum_{j=1}^\infty\int_0^t (t-s)^{-2\alpha}|(I-P_{N_\mu^\lambda})\Pi_2\mathcal{S}_\mu^\lambda(t-s)\mathcal{I}_\mu \Phi(s)e_j\|_{H^{-1}(\lambda)}^2 ds
\\
&=\sum_{j=1}^\infty\int_0^t (t-s)^{-2\alpha}|(A-\lambda)^{-1/2}(I-P_{N_\mu^\lambda})\Pi_2\mathcal{S}_\mu^\lambda(t-s)\mathcal{I}_\mu \Phi(s)e_j|_H^2 ds,\end{aligned}$$
Expanding the squares, the right hand side is 
$$ \Lambda_2(t)\leq \sum_{k=N_\mu^\lambda+1}^\infty\sum_{j=1}^\infty \int_0^t (t-s)^{-2\alpha}\langle \Phi(s)e_j,\mathcal{I}_\mu^*\mathcal{S}_\mu^{\lambda*}(t-s)\Pi_2^*(I-P_{N_\mu^\lambda}^*)(A-\lambda)^{-1/2}e_k\rangle_H^2 ds.
$$

Since 
$$\begin{aligned}&\langle \mathcal{I}_\mu^*\mathcal{S}_\mu^{\lambda*}(t-s)\Pi_2^*(I-P_{N_\mu^\lambda}^*)(A-\lambda)^{-1/2}e_k,e_j\rangle_H^2\\
&=\begin{cases}
-(\alpha_k+\lambda)^{-1/2}f_\mu^\lambda(k)'(t-s),\quad \text{ if } k=j>N_\mu^\lambda,\\ 0\quad\text{ otherwise },
\end{cases}
\end{aligned}$$
and by \eqref{zeta4est4},
$$(\alpha_k+\lambda)^{-1/2}|f_\mu^\lambda(k)'(t-s)|\leq C(\alpha_k+\lambda)^{-1/2}\mu^{-1}.$$

By \eqref{youngused} and the assumptions on $\alpha_k$, for any $\eta\in(0,\eta_0)$ we may find $c_\eta>0$ such that
$$\sum_{k=N_\mu^\lambda+1}^\infty (\alpha_k+\lambda)^{-1}|f_\mu^\lambda(k)'(t-s)|^2\leq \frac{1}{\mu^2}\sum_{k=1}^\infty \frac{1}{(\alpha_k+\lambda)}\leq \frac{c_\eta}{\mu^2} \lambda^{-\eta},$$

where $C_\eta$ is independent of $T$. Since $\int_0^t (t-s)^{-2\alpha}ds$ is finite and uniformly bounded for any $t\in[0,T]$, we deduce that, for some $c_\eta>0$ depending only on $\eta$ and $T$,
$$\Lambda_2(t)\leq \frac{c_\eta}{\mu^2}\lambda^{-\eta}\|\Phi\|_t^2.$$
By the BDG inequality,
$$\mathbb{E}|(I-P_{N_\mu^\lambda})\Pi_2\Gamma_\alpha^{\mu,\lambda}(t)|_{H^{-1}(\lambda)}^p\leq \frac{c_\eta}{\mu^p}\lambda^{-\eta p/2}\mathbb{E}\|\Phi\|_t^p.$$
This completes the proof for the $\zeta=1$ case. The case $\zeta=0$ is analogous, and essentially repeats the computations in the high-modes part of $\zeta=1$, using \eqref{zeta0est2} in place of \eqref{zeta4est4}

The proof regarding evaluation at a stopping time $t\wedge\tau$ is easy: one may just set $\Phi=0$ on $[t\wedge \tau,t]$ and use the estimates we already obtained.

\end{proof}

\begin{theorem}\label{theorem2.11} Fix $T>0$.
Recall $\Gamma^{\mu,\lambda}$ defined in \eqref{Gammamulambda}. Then for any $p\geq 2$ and any $\eta\in(0,\eta_0)$, we may find $c_{\eta,p,\mu}>0$ such that for all $\lambda\geq 0,$ for both $\zeta=0$ and $\zeta=1,$
\begin{equation}\label{1.38lasttheorem}
\mathbb{E}\sup_{t\in[0,T]}|\Pi_1\Gamma^{\mu,\lambda}(t)|_H^p\leq c_{\eta,p,\mu}\lambda^{-\frac{\eta p}{2}} \mathbb{E}\|\Phi\|_T^p .
\end{equation}
The constant $c_{\eta,p,\mu}$ depends on $\eta$, $p$, $\mu$ and $T$ but is independent of $\lambda>0$.

Moreover, for any stopping time $\tau$ adapted to the filtration $(\mathcal{F}_t)_{t\geq 0}$, we have 
\begin{equation}\label{tau1.38lasttheorem}
\mathbb{E}\sup_{t\in[0,T\wedge\tau]}|\Pi_1\Gamma^{\mu,\lambda}(t)|_H^p\leq c_{\eta,p,\mu}\lambda^{-\frac{\eta p}{2}} \mathbb{E}\|\Phi\|_{T\wedge\tau}^p .
\end{equation}
\end{theorem}

\begin{remark}
    A crucial observation shall be made here: the estimate \eqref{1.38lasttheorem} involves the norm of the first component $H\times\{0\}\subset \mathcal{H}$, so that changing the norm of $H^{-1}$ into $H^{-1}(\lambda)$ has no quantitative impact on this estimate. When we solve the stochastic wave equation, as the drift and diffusion coefficients depend only on $u$ but not on $\partial_t u$, we only need to establish uniqueness in law of $u$, then uniqueness in law of $\partial_t u$ follows from integration. Therefore we only need to derive estimates in the first component $H$ of $\mathcal{H}$. The reason why we insist on changing the norm of $H^{-1}$ to $H^{-1}(\lambda)$ is that, when using the factorization formula of stochastic convolutions, the latter choice can generate the desired $\lambda^{-\eta/2}$ coefficient, but the original choice does not give us this coefficient. This will be clear from the proof below. 
\end{remark}

\begin{proof} In the case $\zeta=1,$
    by the stochastic factorization formula (\cite{da2014stochastic},Section 5.3.1), 
\begin{equation}
    \label{factorization}
\begin{aligned}
    \Gamma^{\mu,\lambda}(t)&=\frac{\sin(\alpha\pi)}{\pi}\int_0^t (t-s)^{\alpha-1}\mathcal{S}_\mu^\lambda(t-s)  [\begin{pmatrix} I\\0\end{pmatrix} \Pi_1\Gamma_\alpha^{\mu,\lambda}(s)\\&
    +\begin{pmatrix} 0\\ P_{N_\mu^\lambda}\end{pmatrix} P_{N_\mu^\lambda}\Pi_2\Gamma_\alpha^{\mu,\lambda}(s)+\begin{pmatrix} 0\\ (1-P_{N_\mu^\lambda})\end{pmatrix} (1-P_{N_\mu^\lambda})\Pi_2\Gamma_\alpha^{\mu,\lambda}(s)]ds.\end{aligned} \end{equation}

Now we bound the $L^p$-th norm of $\Gamma^{\mu,\lambda}(t)$ term by term. We start with the first term, choose $\alpha>0$ such that $p>\frac{1}{\alpha}$ and apply Hölder's inequality:
\begin{equation}\label{firstexpanse}\begin{aligned}
&\mathbb{E}\sup_{t\in[0,T]}\left|\int_0^t (t-s)^{\alpha-1}\Pi_1\mathcal{S}_\mu^\lambda(t-s)\begin{pmatrix} I\\0\end{pmatrix} \Pi_1\Gamma_\alpha^{\mu,\lambda}(s)ds\right|_H^p\\
&\leq c_{\eta,p} 
\left(\int_0^T s^{(\alpha-1)p/(p-1)}\|\Pi_1\mathcal{S}_\mu^\lambda(s)\begin{pmatrix} I\\0\end{pmatrix}\|_{\mathcal{L}(H)}^{p/(p-1)}ds
\right)^{p-1} \mathbb{E}\int_0^T |\Pi_1\Gamma_\alpha^{\mu,\lambda}(s)|_H^p ds\\&
\leq c_{\eta,p} \lambda^{-\frac{\eta p}{2}}\mathbb{E}\|\Phi\|_T^p ,
\end{aligned}\end{equation}
where in the second inequality we use Lemma \ref{lemma 1.4} and Proposition \ref{proposition 1.8}, noting that the estimate in Proposition \ref{proposition 1.8} is valid for all $t\in[0,T]$ with the same numerical constant. Thus the constant $c_{\eta,p}$ here depends on $T$ but is finite for any $T>0$.

For the second term
$$\mathbb{E}\sup_{t\in[0,T]}\left|\int_0^t (t-s)^{\alpha-1}\Pi_1\mathcal{S}_\mu^\lambda(t-s)\begin{pmatrix} 0\\P_{N_\mu^\lambda}\end{pmatrix} P_{N_\mu^\lambda}\Pi_2\Gamma_\alpha^{\mu,\lambda}(s)ds\right|_H^p
,$$
we use Lemma \ref{lemma1.5!} and use Proposition \ref{proposition1.10!} (1).

For the third term 
$$\mathbb{E}\sup_{t\in[0,T]}\left|\int_0^t (t-s)^{\alpha-1}\Pi_1\mathcal{S}_\mu^\lambda(t-s)\begin{pmatrix} 0\\I-P_{N_\mu^\lambda}\end{pmatrix} (I-P_{N_\mu^\lambda})\Pi_2\Gamma_\alpha^{\mu,\lambda}(s)ds\right|_H^p,$$
we use Lemma \ref{lemma1.7!} and use Proposition \ref{proposition1.10!} (2). Notice that we are working with the Hilbert space $H^{-1}(\lambda)$ rather than $H^{-1}$. The proof in the case $\zeta=0$ is very similar and indeed easier, so we omit the details. For the second and third term, the leading constant in the estimate all depends on $T$ but is finite for any $T>0$.

The proof for the stopping time version $\mathbb{E}\sup_{t\in[0, T\wedge\tau]}$ is exactly the same, where we just integrate over $[0,T\wedge \tau]$, and use the stopping time version of Proposition \ref{proposition1.10!}.
\end{proof}

\section{Existence and uniqueness of solutions}\label{Section3}

The proof of well-posedness in this section is a bit technical due to difficulties of approximating a continuous map by Lipschitz maps in an infinite dimension setting. The reader may first read Section \ref{alternative} where we work under a slightly stronger assumption, and consequently we have a much shorter proof with the main ideas well illustrated. The general case follows similar lines, but a lot more care has to be taken.

While we prove existence and uniqueness to the stochastic wave equation \eqref{stract}, the value of $\mu>0$ is irrelevant. Thus we suppress $\mu$ from the notation and simply write the solution as $Z(t)=(u(t),v(t))$.

\subsection{Proof of weak existence: step 1}\label{3.111}
For weak existence of solutions to hyperbolic SPDEs on Hilbert space, we cannot use the standard approach in \cite{gkatarek1994weak}, \cite{da2014stochastic}, \cite{gyongy1998existence} that are designed for the parabolic case, because the semigroup $\mathcal{S}_\mu(t)$ of the stochastic wave equation is not a compact operator. Thus we choose to prove weak existence directly. 

First note that since $G$ is uniformly non-degenerate and $B$ has linear growth, we may use Girsanov transform to construct the drift $B$, so we may first without loss of generality assume $B=0$. See \cite{masiero2017well}, Remark 2.1 and references therein for a detailed discussion on applying Girsanov transform for SPDEs in infinite dimensions.

We begin with a standard approximation procedure\footnote{This approximation is used as an illustration but will not be used in the formal proof to be given.}. Consider $P_n$, the projection of $H$ onto the eigenspace spanned by the first $n$ eigenvectors $\{e_1,\cdots,e_n\}$, and consider truncated coefficients $$G^n(t,x):=P_nG(t,P_nx).$$
Now we consider the following SPDEs 

\begin{equation}\label{}
\mu \frac{\partial^2 u^n(t)}{\partial t^2}=Au^n(t) -\zeta \frac{\partial u^n(t)}{\partial t} +G^n(t,u^n(t)) \frac{dW_t}{dt},
\end{equation}
with initial value $(P_nu_0,P_nv_0)\in \mathcal{H}_1$. The initial value will converge to $(u_0,v_0)$ and thus does not matter now. We focus on the stochastic integration, which is given in mild formulation as
$$\int_0^t \Pi_1\mathcal{S}_\mu(t-s)\mathcal{I}_\mu G^n(s,u^n(s))dW(s).$$
First using the mild formulation and Gronwall's inequality one may get an estimate 
$$\mathbb{E}[\sup_{t\in[0,T]}|u^n(t)|^p]\leq M_p$$ for some $n$-independent constant $M_p$. We now wish to establish compactness. We need the following two results from \cite{cerrai2006smoluchowski2}: abbreviating, for any $z\in L^p(\Omega,L^p([0,T];\mathcal{H}_0)),$ $$\Gamma_\mu(z)(t)=\int_0^t\mathcal{S}_\mu(t-s)\mathcal{I}_\mu G(s,\Pi_1z(s))dW(s),$$
\begin{proposition}\label{proposition}
    Fix $T>0$, given any $p>4$ consider $z\in L^p(\Omega,L^p([0,T];\mathcal{H}_0))$ and any $\delta<\eta_0-2/p$, we 
    have, for the case of damped stochastic wave equation $\zeta=1$: 
    \begin{equation}
        \sup_{\mu>0} \mathbb{E}[\Pi_1|\Gamma_\mu (z)(t)|_{\mathcal{C}([0,T];H^\delta)}]\leq C_{p,T} (1+\mathbb{E}\int_0^T |\Pi_1 z(t)|_H^p dt),
    \end{equation}
    where $C_{p,T}$ does not depend on $\mu$. The parameter $\eta_0$ is the one given in Assumption $\mathbf{H}1$.
    
    For the case $\zeta=0$, i.e. with no damping term $\partial_t u$, we have the same estimate but with a constant $C_{p,T,\mu}$ that is $\mu$-dependent.
\end{proposition}
The proof of this result in the case $\zeta=1$ is given in \cite{cerrai2006smoluchowski2}, Proposition 3.1. The Laplacian on $[0,1]$ was considered there, but generalization to the operator $A$ satisfying $\mathbf{H}1$ is straightforward. This Proposition can also be proved via modifying the proof of Theorem \ref{theorem2.11ext}. The proof for the $\zeta=0$ version is even easier as the constant can be $\mu$-dependent. We omit the details but remark that the proof can be written via slightly modifying the proof of Theorem \ref{theorem2.11}. We only need to replace all the $H$-norms by $H^\delta$ norms and use the summability of eigenvalues $\sum_k \frac{1}{(\alpha_k)^\omega}$ for appropriate powers of $\omega$.  

The following Proposition is also a useful tool in establishing tightness:
\begin{proposition}\label{propsaulte}
    For any $\epsilon$ sufficiently small, $\rho\in(0,\epsilon)$ and $p\geq 1$, we can find a constant $c=c(\epsilon,\rho,p)$ such that 
$$\sup_{\mu\geq 0} \mathbb{E}|\Pi_1\Gamma_\mu(z)(t)-\Pi_1\Gamma_\mu(z)(s)|^p\leq c|t-s|^{\rho p/2} (1+\mathbb{E}[\Pi_1 z]^p_{\mathcal{C}([0,T];H^{\epsilon})}),
$$ 
    given $z\in L^p(\Omega;\mathcal{C}([0,T];\mathcal{H}_{\epsilon})$ and $t,s\in[0,T]$.
\end{proposition}

The proof of this result is given in \cite{cerrai2006smoluchowski2} as a lemma after the proof of Proposition 3.1. However, one should be careful that this result is only proved for random field valued SPDEs defined on $[0,1]$, and Sobolev embedding is used in its proof as an essential procedure. Thus we can use this result for random field SPDEs, and the proof of existence will follow. But in our setting of abstract Hilbert space mappings, we cannot use this result so we choose to give an independent proof.

To deduce compactness in $\mathcal{C}([0,T];H)$, the standard procedure is to use the infinite dimensional Arzela-Ascoli theorem. In this direction, one need to show (1) tightness at each time $t$, and (2) uniform continuity of sample paths. Now (1) can be justified with high probability thanks to Proposition \ref{proposition} since $\Pi_1\Gamma_\mu(z)(t)$ is uniformly bounded in $H^\delta$ for each $t$, and $H^\delta$ embeds compactly into $H$ (also note that we assume the initial value $(u_0,v_0)\in\mathcal{H}_1$). If one assumes Proposition \ref{propsaulte}, then  condition (2) can be justified and we are done. However, this is not the case for hyperbolic SPDEs on Hilbert space (though for parabolic SPDEs this is true.) The reason is that if we consider the undamped wave equation, the solution will be sines and cosines. Under the evolution of wave semigroup, we have a perfect control of the norms of these solutions, but our control on the time derivative of the solution is rather poor-- they will have very fast periodic oscillation! For random field stochastic wave equation, we can still get Hölder regularity in time increments thanks to the fact that the diffusion coefficient $g(t,u(t,x))$ has higher spatial regularity than merely being an $L^2([0,1])$ function. Unfortunately, this is not assumed anywhere in our paper. Therefore, the Arzela-Ascoli argument does not apply and we will turn back in Section \ref{weakexistence} to finally establish the proof of weak existence, using estimates derived in Section \ref{weakunique}.

\subsection{Lipschitz approximation of continuous maps}\label{767767}
In the later proofs we will frequently consider Lipschitz approximations of Hölder continuous mappings. Such approximations are well-understood for real-valued functions, but our target space $H$ has infinite dimensions. Fortunately, we can still find an approximation procedure, such that the approximating sequences have good properties inherited from the original function. We begin with the following Proposition, whose proof is taken from Alex Ravsky \cite{610367}.

\begin{proposition}
    Given $V$ a separable Hilbert space, $V'$ be a normed space and $f:V\to V'$ a continuous map. Then we can find a sequence $\{f_n\}$ of bounded Lipschitz maps: $V\to V'$ that converge uniformly to $f$ on any compact subset of $V$. 
\end{proposition}

\begin{proof} We only need to consider $V=\ell^2$. Write $\mathbb{Z}/n=\{x\in\mathbb{R}:nx\in\mathbb{Z}\}$. First define $g_n^*:(\mathbb{Z}/n)^n\to V'$ via $g_n^*(x)=f(x)$ for $|f(x)|\leq n$ and $g_n^*(x)=nf(x)/|f(x)|$ if $|f(x)|\geq n$. Then extend $g_n^*$ to a map $f_n^*:\mathbb{R}^n\to V'$ via affine interpolation. More precisely, consider a map $\lfloor\cdot\rfloor_n:\mathbb{R}\to\mathbb{Z}/n$ via $\lfloor x\rfloor_n=\lfloor nx\rfloor /n$ for $x\in\mathbb{R}$. Given any $x=(x_i)\in\mathbb{R}^n$, $x$ is contained in $\mathcal{C}_x=\prod_{i=1}^n [\lfloor x_i\rfloor_n;\lfloor x_i\rfloor_n+1/n]$. Now define $f_n^*(x)=\sum\{\lambda(d,x)g_n^*(d):d\text{ is a vertex of $C_x$}\},$ with $\lambda(d,x)=\prod_{i=1}^n \mu_i(d_i,x_i)$ and $\mu_i(d_i,x_i)=n(x_i-\lfloor x_i\rfloor_n)$ when $d_i=\lfloor x_i\rfloor_n+1/n$ and $\mu_i(d_i,x_i)=1-n(x_i-\lfloor x_i\rfloor_n)$ when $d_i=\lfloor x_i\rfloor_n$. Then we have $f_n^*(x)$ lies in $\operatorname{Conv}\{g_n^*(d):d\text{ is a vertex of the cube } C_x\}$.

Finally put $f_n=f_n^*p_n$ where $p_n:\ell_2\to\mathbb{R}^n$ is the orthogonal projection.

The proof that $f_n$ is Lipschitz continuous, and that $f_n$ converges to $f$ on compacts, are given in \cite{610367}.    
\end{proof}

Now for any $t>0$, we can apply the above Proposition to the continuous map $G(t,x):H\to\mathcal{L}(H_0,H)$ where $H$ is a separable Hilbert space and $\mathcal{L}(H_0,H)$ is a normed space. When the drift $B$ is also continuous, we can apply this proposition to $B(t,x):H\to H$.

When $f$ has some more properties, $f^*_n$ will inherit the same property when $n$ is large, and so does $f_n$. Say if we consider $f=G(t,x)$ for the diffusion coefficient $G$, then as $G$ is uniformly non-degenerate, $f_n^*$ is also uniformly non-degenerate for $n$ large, and so does $f_n$. This is because $f_n^*(x)$ is a weighted sum of $g_n^*(d)$ for $d$ vertices of $C_x$, and when $n$ is large, by Hölder continuity of $G$, the difference of the values of $G$ at different sites of $C_x$ becomes negligible. Use also the fact that given an operator $A$ on $H$ with very small operator norm, then $I+A$ has a right inverse on $H$. This proves the claim. 

Moreover, assume $G(t,x)$ is $\beta>0$-Hölder continuous in $x$, then $f^n$ will also be $\beta$-Hölder with Hölder constant uniform in $n$, at least for any $x,y\in H$ with $|x-y|>d_n$, Here $d_n$ is some constant that tends to $0$ as $n\to\infty$.  This is because, when $x$ and $y$ lie in the same cube $C_x$, then $|f_n(x)-f_n(y)|$ is at most $C(\frac{1}{n})^\beta$, while if they belong to different cubes $C_x$ and $C_y$, the difference $|f_n(x)-f_n(y)|$ can be well approximated by the difference of $g_n^*(d)$ when $d$ ranges over vertex of $C_x$ and $C_y$ respectively, and the latter is bounded by $|x-y|^\beta$ up to some constant. Thus whenever $|x-y|$ is much larger than $\frac{1}{n}$, which is the typical size of the cube $C_x$, we can regard $f^*_n(x)$, and hence $f_n(x)$, as $\beta$-Hölder functions with uniform in $n$ Hölder constants.

\subsection{Proof of weak uniqueness}\label{weakunique}

Now we prove the uniqueness of the weak mild solution under assumptions $\mathbf{H}1$ through $\mathbf{H}5$. We also assume that\footnote{In Section \ref{weakexistence} we will show that Assumption $\mathbf{H}6$ is implied by Assumptions $\mathbf{H}1$ to $\mathbf{H}5$.}

$\mathbf{H}6$: for the given initial value $(u_0,v_0)\in \mathcal{H}_1$, there exists a weak mild solution to \eqref{stract}. Moreover, for any $\epsilon>0$ there exists a compact subset $K\subset H$ such that with probability at least $1-\epsilon$, $u(t)$ never leaves $K$ for $t\in[0,T]$.
The proof in this section is inspired by \cite{kulik2020well}.

We first make several important observations:
\begin{enumerate}
    \item From the mild formulation \eqref{eq2.4geg} of the SPDE \eqref{stract}, and that $\Pi_1 Z(s)=u(s),$ it is clear that we only need to prove weak uniqueness for the first component $u(s)$ of $Z(s)=(u(s),v(s))$. Once that is proved, weak uniqueness for the whole process $Z(s)$ follows from the mild formulation \eqref{eq2.4geg}.

\item  By the non-degeneracy assumption $\mathbf{H}4$ of the diffusion coefficient $G$, and the linear growth assumption of $B$ in $\mathbf{H}5$, we see that once weak uniqueness is proved for the SPDE with $B=0$, weak uniqueness for general $B$ follows from an application of Girsanov theorem. Therefore in the following proof we assume $B=0$.

\end{enumerate}

As shown in the previous section, we can find a sequence of operators $G^n(t,x)$, which are Lipschitz in $x$, and moreover they are uniformly non-degenerate and $\beta$-Hölder in $x$ with $n$-independent coefficients.

Now consider the following three SPDEs, with the same initial value $(u_0,v_0)\in \mathcal{H}_1$:
\begin{equation}\label{unieqn1}
\mu \frac{\partial^2 u(t)}{\partial t^2}=Au(t) -\zeta \frac{\partial u(t)}{\partial t} +G(t,u(t)) \frac{dW_t}{dt},
\end{equation}
\begin{equation}\label{unieqn2}
\mu \frac{\partial^2 u^n(t)}{\partial t^2}=Au^n(t) -\zeta \frac{\partial u^n(t)}{\partial t} +G^n(t,u^n(t)) \frac{dW_t}{dt},
\end{equation}
\begin{equation}\label{unieqn3}
\mu \frac{\partial^2 \widetilde{u}^n(t)}{\partial t^2}=A\widetilde{u}^n(t) -\zeta \frac{\partial \widetilde{u}^n(t)}{\partial t} +\lambda(u(t)-\widetilde{u}^n(t))1_{t\leq\tau}
+G^n(t,\widetilde{u}^n(t)) \frac{dW_t}{dt},
\end{equation}
where $\tau$ is some stopping time adapted to the filtration $(\mathcal{F}_t)_{t\geq 0}$ and $\lambda>0$ is a constant to be determined later in the proof.

In this proof we aim at proving weak uniqueness of \eqref{unieqn1}. By Assumption $\mathbf{H}6$, we have constructed a weak mild solution to \eqref{unieqn1} on some filtered probability space. Since $G^n$ is Lipschitz continuous, \eqref{unieqn2} and \eqref{unieqn3} have strong uniqueness, so we can solve \eqref{unieqn2} and \eqref{unieqn3} on the same probability space as \eqref{unieqn1} and assume they are driven by the same cylindrical white noise $W$.

Given any compact subset $K\subset H$ containing the initial value $u_0$, denote by 
$$\Delta_K^n:=\sup_{t\geq 0}\sup_{y\in K}|G^n(t,y)-G(t,y)|_{\mathcal{L}(H_0,H)}.$$
Since $G^n(t,x)$ converges to $G(t,x)$ for each $x\in H$, that $K$ is compact and $G(t,x)$ is uniformly continuous 
in $x$, we deduce that $\lim_{n\to\infty}\Delta_K^n=0$ since all the estimates are time independent.
Now define two stopping times
$$
\tau_K^n:=\inf\{t\geq 0:|u(t)-\widetilde{u}^n(t)|\geq 2\Delta_K^n\},
$$
and 
$$
\theta_K:=\inf\{t\geq 0:u(t)\notin K\}.
$$
In the reset of the proof we consider $$\tau=\tau_K^n\wedge\theta_K,\quad \lambda=(\Delta_K^n)^{\gamma-1},$$ where $\gamma$ will be specified later. We may slightly enlarge $\Delta_k^n$, but still assume $\lim_{n\to\infty}\Delta_K^n=0,$ such that $\Delta_K^n$ is much larger than the length scale for $G^n$ to be uniformly $\beta$-Hölder continuous in $x$, as discussed at the end of Section \ref{767767}.

Now we can apply Girsanov theorem to estimate $u^n(t)-\widetilde{u}_n(t)$. Denote by $d_{TV}$ the total variation distance on $\mathcal{C}([0,T];H)$ and $H_{rel}(\cdot\mid\cdot)$ the relative entropy between two probability measures on $\mathcal{C}([0,T];H)$. By Pinsker's inequality, 
$$d_{TV}(\operatorname{Law}(u^n|_{[0,T]}),\operatorname{Law}(\widetilde{u}^n|_{[0,T]}))\leq \sqrt{2H_{rel}(\operatorname{Law}(u^n|_{[0,T]})\mid \operatorname{Law}(\widetilde{u}^n|_{[0,T]}))}.$$
To compute the relative entropy, we use Girsanov transform, noting that the SPDE solved by $\widetilde{u}^n$ \eqref{unieqn3} can be obtained by the SPDE solved by $u^n$ \eqref{unieqn2} via adding a drift to the Brownian sample path
\begin{equation}
    (dW_t)_{t\geq o}\mapsto (dW_t+(G^n(t,u^n(t)))^{-1}\lambda(u(t)-{u}^n(t))1_{t\leq\tau}dt)_{t\geq 0}.
\end{equation}
Since $(G^n)^{-1}$ is a bounded operator, recalling the definition of the stopping time $\tau$ and $\tau_K^n$, and the choice of $\lambda$, we deduce that
\begin{equation}\label{totalvariation}
    d_{TV}(\operatorname{Law}(u^n|_{[0,T]}),\operatorname{Law}(\widetilde{u}^n|_{[0,T]}))\leq CT^{1/2}(\Delta_K^n)^\gamma.
\end{equation}

Now we estimate $u(t)-\widetilde{u}^n(t)$ via pathwise arguments. For $0\leq  \tau \leq t$ we have the almost sure bound, whenever $\Delta_K^n\leq 1:$  in the following we use $|\cdot|_{\mathcal{L}(H_0,H)}$ for the norm of $G(t,x)$:
\begin{equation}\label{triangle}
\begin{aligned}
    |G(t,u(t))-G^n(t,\widetilde{u}^n(t))|&\leq |G(t,u(t))-G^n(t,u(t))|+|G^n(t,u(t))-G^n(t,\widetilde{u}^n(t))|
    \\&\leq \Delta_K^n+ M(2\Delta_K^n)^\beta\leq C(\Delta_K^n)^\beta,
\end{aligned}\end{equation}
where in the second inequality, for the first term we use that before time $\tau$, $u(t)$ stays in $K$ and $|G-G^n|\leq \Delta_K^n$ on $K$; and for the second term we use the Hölder continuity of $G^n$ and that $|u-\widetilde{u}^n|\leq\Delta_K^n$ for $t\in[0,\tau]$. Here we also use that $\Delta_K^n$ is much larger than the length scale for $G^n$ to be $\beta$-Hölder with uniform in $n$ constants.

By the mild formulation \eqref{generalpi1} and the definition of $\mathcal{S}_\mu^\lambda$,

$$ u(t)-\widetilde{u}^n(t)=\int_0^t \Pi_1 \mathcal{S}_\mu^\lambda(t-s)\mathcal{I}_\mu [G(s,u(s))-G^n(s,\widetilde{u}^n(s))]dW_s,\quad 0\leq  t\leq \tau.
$$
Applying the maximal inequality \eqref{tau1.38lasttheorem}, Markov's inequality and estimate \eqref{triangle}, we deduce that for each $\eta\in(0,\eta_0)$ and $m\in\mathbb{N}_+$ we can find some $C=C_{\eta,m,\mu}>0$ such that, for each $R>0$, [Recall the choice $\lambda=(\Delta_K^n)^{\gamma-1}$]
\begin{equation}\label{869}
    \mathbb{P}(\sup_{t\in[0,T\wedge\tau]}|u(t)-\widetilde{u}^n(t)|\geq C(\Delta_K^n)^{\frac{1-\gamma}{2}\eta+\beta}R)\leq\frac{C}{R^m}.
\end{equation}

By our assumption $\beta>1-\frac{\eta_0}{2}$, we may choose $\gamma>0$ sufficiently small and $\eta<\eta_0$ sufficiently close to $\eta_0$ such that $\frac{1-\gamma}{2}\eta+\beta>2\chi+1>1$ for some $\chi>0$. Now we take $R=(\Delta_K^n)^{-\chi}$ so that 
\begin{equation}\label{870}
    \mathbb{P}(\sup_{t\in[0,T\wedge\tau]}|u(t)-\widetilde{u}^n(t)|\geq C(\Delta_K^n)^{1+\chi})\leq C(\Delta_K^n)^{m\chi}.
\end{equation}

Consider the event 
$$\Omega_\nu:=\{\sup_{t\in[0,\tau]}|u(t)-\widetilde{u}^n(t)|>C (\Delta_K^n)^{1+\chi}\},$$ then on $\Omega\setminus\Omega_\nu$, one must have $\|u(t)-\widetilde{u}^n(t)|\leq \Delta_K^n$ for $t\in[0,\tau]$, then by continuity of trajectories this must hold all the way up to $T\wedge\theta_K$. In other words, on $\Omega\setminus\Omega_\nu$,
$$\theta_K\wedge T\leq \tau_K^n.$$ This implies, for any $\kappa>0$, as $n$ gets large $\Delta_K^n\to 0$, we have
\begin{equation}\label{expexpepx}
    \mathbb{P}(\sup_{t\in[0,T\wedge \theta_K]}|u(t)-\widetilde{u}^n(t)|>\kappa)\to 0,\quad n\to\infty.
\end{equation}

We can now finish the proof of uniqueness in law. Consider $F$ a real-valued, bounded continuous function defined on $\mathcal{C}([0,T]; H)$. Once we can show that
\begin{equation}\label{anyany}
    \mathbb{E}[F( u|_{[0,T]})-F( u^n|_{[0,T]})]\to 0,\quad n\to\infty,
\end{equation}
uniqueness in law for $u$ holds. This is because the law of each $u^n$ is well-defined (note that $G^n$ is Lipschitz). If the law of $u|_{[0,T]}$ were not unique, any weak solution of $u$ should satisfy all the above estimates, so any weak solution $u$ should satisfy \eqref{anyany}, but this forces $\mathbb{E}[F(u|_{[0,T]})]$ to be the same for each weak solution $u$ given any test function $F$, and hence the law of $u$ is uniquely determined. 

 To check \eqref{anyany}, note first that \eqref{totalvariation} implies
\begin{equation}\label{anyany2}
    \mathbb{E}[F(\widetilde{u}^n|_{[0,T]})-F( u^n|_{[0,T]})]\to 0,\quad n\to\infty,
\end{equation}
and \eqref{expexpepx} implies that on $\{\theta_K\geq T\}\cap(\Omega\setminus\Omega_\nu),$ we have that $\widetilde{u}^n|_{[0,T]}$ converges in probability to $u_{[0,T]}$ as $n\to\infty$. Since $\lim_{n\to\infty}\mathbb{P}(\Omega_\nu)=0,$ we conclude that 
\begin{equation}\label{anyanys}
   \lim\sup_{n\to\infty} \left|\mathbb{E}[F( u|_{[0,T]})-F( u^n|_{[0,T]})]\right|\leq 2\sup_x |F(x)|\mathbb{P}(\theta_K\leq T).
\end{equation}
By Assumption $\mathbf{H}6$, we can find compact sets $K\subset H$ such that $\mathbb{P}(\theta_K\leq T)$ is arbitrarily small. As the left hand side of \eqref{anyanys} does not depend on $K$, we have established 
\eqref{anyany}, and finally finished the proof of weak uniqueness for $u(t)$.

\subsection{Proof of weak existence: step 2}\label{weakexistence}
In this section we prove that Assumption $\mathbf{H}6$ holds assuming Assumptions $\mathbf{H}1$ to $\mathbf{H}5$, thus establishing weak existence of \eqref{stract}.

In this proof it will be convenient to work with the following Wasserstein distance on $\mathcal{C}([0,T];H)$: define a distance $d$ on $\mathcal{C}([0,T];H)$ as: for any $\varphi_1,\varphi_2\in \mathcal{C}([0,T];H),$ 
$$d(\varphi_1,\varphi_2)=\sup_{0\leq t\leq T}|\varphi_1(t)-\varphi_2(t)|_H\wedge 1,$$
then for any probability measure $\mu$ and $\nu$ on $\mathcal{C}([0,T];H)$, we define the Wasserstein distance of $\mu$ and $\nu$ as: 
$$W(\mu,\nu)=\inf_{(X,Y)}\mathbb{E}[ d(X,Y)],$$
where $(X,Y)$ ranges over all couplings of $\mu$ and $\nu$, with $\operatorname{Law}(X)=\mu$ and $\operatorname{Law}(Y)=\nu$. Note that the metric $d$ we defined is equivalent to the standard metric on $\mathcal{C}([0,T];H)$.

Recall that we have constructed Lipschitz approximations $G^n$, and we set out to compare the Wasserstein distance between 
\begin{equation}\label{unieqn200}
\mu \frac{\partial^2 u^n(t)}{\partial t^2}=Au^n(t) -\zeta \frac{\partial u^n(t)}{\partial t} +G^n(t,u^n(t)) \frac{dW_t}{dt},
\end{equation}
and \begin{equation}\label{unieqn20}
\mu \frac{\partial^2 u^m(t)}{\partial t^2}=Au^m(t) -\zeta \frac{\partial u^m(t)}{\partial t} +G^m(t,u^m(t)) \frac{dW_t}{dt},
\end{equation}
with the same initial data $(u_0,v_0)$. 

First note that, for every $\epsilon>0$, we can find a compact set $K_\epsilon\subset H$ such that for each $n$, with probability at least $1-\epsilon$, $u^n(t)$ will stay in $K_\epsilon$ throughout $t\in[0,T]$. This is guaranteed by Proposition \ref{proposition}, the compact embedding of $H^\delta$ into $H$, and Markov's inequality. Let $\tau_{K_\epsilon}^n$ denote the first time $u^n(t)$ leaves $K_\epsilon$, then $\mathbb{P}(\tau_{K_\epsilon}^n< T)\leq \epsilon$.

Now we consider the auxiliary process 
\begin{equation}\label{**!!12306unieqn3}
\mu \frac{\partial^2 \widetilde{u}^m(t)}{\partial t^2}=A\widetilde{u}^m(t) -\zeta \frac{\partial \widetilde{u}^m(t)}{\partial t} +\lambda(u^n(t)-\widetilde{u}^m(t))1_{t\leq\tau}
+G^m(t,\widetilde{u}^m(t)) \frac{dW_t}{dt},
\end{equation} 

Assume that $m,n\geq N$ and consider, for any compact subset $K\subset H$: 
$$c_{N,K}:=\sup_{t\geq 0}\sup_{n\geq N}\sup_{x\in K} |G^n(t,x)-G(t,x)|_{\mathcal{L}(H_0,H)}.$$
By definition $\lim_{N\to\infty} c_{N,K}=0.$ Assume $N$ is already chosen large enough sot that $c_{N,K_\epsilon}\leq \frac{\epsilon}{4}$. Now we set  (abbreviating the subscript $K_\epsilon$ and simply write $K$): $\lambda=c_{N,k}^{\gamma-1}$ and $\tau:=T\wedge \tau_{K_\epsilon}^n\wedge \{t\geq 0:|u^n(t)-\widetilde{u}^m(t)|\geq c_{N,K}\}$, where $\gamma$ takes the same value as in Section \ref{weakunique}. 

Then a total variation estimate shows that 
$$d_{TV}(\operatorname{Law}(u^m|_{[0,T]},\operatorname{Law}(\widetilde{u}^m|_{[0,T]})))\leq C(c_{N,K})^\gamma\leq \epsilon^\gamma.$$

A pathwise argument as in \eqref{869}, \eqref{870} shows, for some $\chi>0$, $\mathfrak{\chi}_0>0,$
\begin{equation}
  \mathbb{P}(\sup_{t\in[0,T\wedge\tau]}|u^n(t)-\widetilde{u}^m(t)|\geq (c_{N,K})^{1+\chi_0})\leq (c_{N,K}) ^\chi\leq\epsilon^\chi. 
\end{equation}

By continuity of trajectories and the fact that $\chi_0>0$, we see that if $|u^n(t)-\widetilde{u}^n(t)|\leq (C_{N,K})^{1+\chi_0}$ for $t\in[0,T\wedge \tau]$, then necessarily $\tau=T\wedge \tau^n_{K_\epsilon}$. By assumption $\mathbb{P}(\tau_{K_\epsilon}^n< T)\leq \epsilon$, so that upon slightly changing the value of $\chi>0$, we can conclude with 
\begin{equation}
  \mathbb{P}(\sup_{t\in[0,T]}|u^n(t)-\widetilde{u}^m(t)|\geq c_{N,K})\leq\epsilon^\chi. 
\end{equation}

By definition of total variation distance, we can find a coupling of $u^m|_{[0,T]}$ and $\widetilde{u}^m|_{[0,T]}$ such that they are identical on an event of probability at least $1-\epsilon^\gamma$, and arbitrary otherwise. From this we can deduce a coupling of $u^n|_{[0,T]}$ and $u^m|_{[0,T]}$: we use the same Wiener process to drive $u^n$ and $\widetilde{u}^m$, and then combine with the coupling of $u^m$ and $\widetilde{u}^m$ to obtain a coupling of $u^n$ and $u^m$. 

Note that the distance $d$ is upper bounded by $1$, thus we have, for some $\omega>0$,
$$W(u^n|_{[0,T]},u^m|_{[0,T]})\leq c_{N,K}+\epsilon^\chi+\epsilon^\gamma+\epsilon \leq C\epsilon^\omega.$$ 

We can restate our finding as follows: for any $\epsilon>0$, we can find $N>0$ such that whenever $n,m\geq N$, we have $W(u^n|_{[0,T]},u^m|_{[0,T]})\leq \epsilon.$
This implies that $u^n|_{[0,T]}$ forms a weakly converging subsequence. Applying Skorokhod embedding theorem, the limit should solve the stochastic wave equation \eqref{stract} by continuity of the coefficient $G$. This establishes weak existence, and hence finishes the proof of Theorem \ref{theorem1.1}. That the solution will stay in a compact subset with high probability again follows from Proposition \ref{proposition}.

\subsection{Alternative proof of existence and uniqueness assuming uniform approximation}\label{alternative}

In this section we give an alternative, shorter proof of weak existence and uniqueness under the following stronger assumption

$\mathbf{H}7$: There exists a sequence of mappings $G^n(t,x):[0,\infty)\times H\to \mathcal{L}(H_0,H)$ that satisfy assumptions $\mathbf{H}3,\mathbf{H}4,\mathbf{H}5$ with uniform in $n$ constants, that each $G^n$ is Lipschitz in $x$ with Lipschitz constant uniform in $t$, and moreover we have the uniform approximation
\begin{equation}
    \lim_{n\to\infty} \sup_{t>0}\sup_{x\in H} |G^n(t,x)-G(t,x)|_{\mathcal{L}(H_0,H)}=0.
\end{equation}

Though assumption $\mathbf{H}7$ looks much stronger, it indeed covers many interesting examples of random field solutions to the stochastic wave equation. Say we consider $H=L^2([0,1];\mathbb{R})$ and $G(t,u)(x)=g(t,u(t,x))$ for any $u\in H$. Assume that for some $\beta\in(\frac{3}{4},1]$ we have
    \begin{equation}
        g(t,u(t,x))=1+|u(t,x)|^\beta,
    \end{equation}
so the only place the noise coefficient $g$ is non-Lipschitz is at $u=0$. We can find a sequence of $g_n$ approximating $g$, each $g_n$ being Lipschitz continuous in $u$, and \begin{equation}\label{funnyguide}\sup_{t\geq 0}\sup_{u\in\mathbb{R}}|g_n(t,u)-g(t,u)|\to 0,\quad n\to\infty.\end{equation} Defining $G^n$ accordingly via $G^n(t,u)v=g_n(t,u(t,x))v(x),x\in[0,1]$ for any $v\in L^\infty([0,1])$, so that the approximating sequence $G^n$ satisfies $\mathbf{H}7$. More generally, one could consider diffusion coefficients $g$ that are Lipschitz continuous away from finitely many irregular points, or consider those coefficients $g$ such that we can find a sequence $g_n$ uniformly approximating $g$ in the sense of \eqref{funnyguide}.

 We follow roughly the same steps as in the previous section: we may assume $B=0$ and consider the following SPDEs

\begin{equation}\label{12306unieqn1}
\mu \frac{\partial^2 u(t)}{\partial t^2}=Au(t) -\zeta \frac{\partial u(t)}{\partial t} +G(t,u(t)) \frac{dW_t}{dt},
\end{equation}
\begin{equation}\label{12306unieqn2}
\mu \frac{\partial^2 u^n(t)}{\partial t^2}=Au^n(t) -\zeta \frac{\partial u^n(t)}{\partial t} +G^n(t,u^n(t)) \frac{dW_t}{dt},
\end{equation}
\begin{equation}\label{12306unieqn3}
\mu \frac{\partial^2 \widetilde{u}^n(t)}{\partial t^2}=A\widetilde{u}^n(t) -\zeta \frac{\partial \widetilde{u}^n(t)}{\partial t} +\lambda(u(t)-\widetilde{u}^n(t))1_{t\leq\tau}
+G^n(t,\widetilde{u}^n(t)) \frac{dW_t}{dt},
\end{equation} 
Under assumption $\mathbf{H}7$ we no longer need to consider the compact subset $K$, and we can use the (globally uniformly) approximating sequence $G^n$ given in $\mathbf{H}7$ . 

Denote by 
$$c_n:=\sup_{t\geq 0}\sup_{x\in H}|G^n(t,x)-G(t,x)|_{\mathcal{L}(H_0,H)}$$ and consider $\tau:=T\wedge \inf\{t\geq 0:|u(t)-\widetilde{u}^n(t)|\geq 2c_n\}$. Then we set $\lambda=(c_n)^{\gamma-1}$ and find some $\eta<\eta_0$ close to $\eta_0$ such that $\frac{1-\gamma}{2}\eta+\beta>2\chi+1>1$ for some $\chi>0$. Then with exactly the same proof as in the previous section, we have

\begin{equation}\label{12315riation}
    d_{TV}(P_nu^n|_{[0,T]},P_n\widetilde{u}^n|_{[0,T]})\leq CT^{1/2}(c_n)^\gamma.
\end{equation}
and 
\begin{equation}
    \mathbb{P}(\sup_{t\in[0,T\wedge\tau]}|u(t)-\widetilde{u}^n(t)|\geq C(c_n)^{1+\chi})\leq C(c_n)^{m\chi}.
\end{equation}
so that for any $\kappa>0$,
\begin{equation}\label{12315expexpepx}
    \mathbb{P}(\sup_{t\in[0,T]}|u(t)-\widetilde{u}^n(t)|>\kappa)\to 0,\quad n\to\infty.
\end{equation}

Combining these arguments, we have established uniqueness of weak mild solution to \eqref{stract} in exactly the same way as in the previous section. 

For weak existence, following (a simplified version of) the argument in Section \ref{weakexistence}, we deduce that $u^n|_{[0,T]}$ forms a Cauchy sequence with respect to the Wasserstein distance, hence has a converging subsequence and the limit is a weak mild solution to \eqref{stract}.

\section{The Smoluchowski-Kramers approximation}\label{Section4}

\subsection{Mass-independent estimates of stochastic integrals} In this section we only focus on the $\zeta=1$ case, i.e. the stochastic wave equation with damping term. 
In proving the small-mass limit of the damped stochastic wave equation, the maximal inequality estimate in Proposition \eqref{proposition 1.8} is not enough because the numerical constant will diverge in the limit $\mu\to 0$.

All the estimates in this section are modifications to those in \cite{salins2019smoluchowski}, Section 6. The main difference is that, in \cite{salins2019smoluchowski}, all the eigenvalues are $-\alpha_k$, but in our case they are $-\alpha_k-\lambda$. We will use the summability of $\alpha_k$ to extract a factor of $\lambda$ out. For that purpose, we have to rewrite almost every proof in \cite{salins2019smoluchowski}, Section 6 instead of quoting them, but the essential ideas are not changed.

In all the following estimates, the constant $C_{\eta,p}$ is finite and may implicitly depend on the terminal time $T>0$, but its precise value is irrelevant to us so we do not discuss its precise dependence on $T$. All the estimates we derive hold uniformly over all $t\in[0,T]$ with the same numerical constant $C_{\eta,p}$.

\begin{proposition}\label{proposition 4.8extra}
Assume that there is some $\eta_0\in(0,1)$ such that for some constant $C>0$, \begin{equation}\label{summoperator4.8}
\sum_n \frac{1}{\alpha_n^{1-\eta}}<\infty, \quad \frac{1}{C} k^{\frac{1}{1-\eta_0}}\leq \alpha_k\leq Ck^{\frac{1}{1-\eta_0}} \quad \forall k\in\mathbb{N}_+
\end{equation}
for each $\eta\in(0,\eta_0)$. Then for each fixed $\lambda>0$ we can find some $\mu_0=\mu_0(\lambda)>0$ as a function of $\lambda$ such that for any $\mu\in(0,\mu_0(\lambda)]$,
for each $p\geq 2$ and each $\eta\in(0,\eta_0),$ 
\begin{equation}
    \mathbb{E}[|\Pi_1\Gamma_\alpha^{\mu,\lambda}(t)|_H^p]\leq C_{\eta,p}\mathbb{E}\|\Phi\|_t^p (\lambda^{-\frac{\eta p}{2}}+\mu^{(\eta_0-2\alpha)p/2}),\quad\text{ for any }t\in[0,T].
\end{equation} The constant $c_{\eta,p}$ depends on $\eta$, $p$ and $T$ but does not depend on $\lambda$ and $\mu$. By choosing $\mu_0(\lambda)$ small, we may discard the $\mu^{(\eta_0-2\alpha)p/2}$ factor and the estimate does not depend on $\mu$.
\end{proposition}

\begin{proof}
Recall from the proof of Proposition \ref{proposition 1.8} that $$\begin{aligned}
\Lambda_\alpha^{\mu,\lambda}(t):&=\sum_{j=1}^\infty \int_0^t (t-s)^{-2\alpha}|\Pi_1\mathcal{S}_\mu^\lambda(t-s)\mathcal{I}_\mu \Phi(s)e_j|_H^2ds \\
&\leq\int_0^t (t-s)^{-2\alpha}\|\Phi\|_s^2\left(\sum_{k=1}^\infty f_\mu^\lambda(k)^2(t-s,0,1/\mu)
\right)ds,\end{aligned}
$$
and 
$$\sum_{k=1}^\infty f_\mu^\lambda(k)^2(t,0,1/\,u)\leq \sum_{k=1}^{N_\mu^\lambda}e^{-2(\alpha_k+\lambda)t}+\sum_{k=N_\mu^\lambda+1}^\infty \frac{e^{-\frac{t}{2\mu}}}{\mu(\alpha_k+\lambda)}:=J_1(t)+J_2(t).$$
and that
$$N_\mu^\lambda:=\max\{k\in\mathbb{N}:1-4\mu(\alpha_k+\lambda)\geq 0\},$$
we assume that for the fixed $\lambda>0$, $\mu>0$ is chosen sufficiently small such that $\mu\lambda\leq\frac{1}{40}$, that is, we set $\mu_0(\lambda)=\frac{1}{40\lambda}$. Under this choice for $k>N_\mu^\lambda$ one must have $\mu\alpha_k\geq \frac{9}{40}$ and hence $\alpha_k\geq 9\lambda$ for all $k>N_\mu^\lambda$. Thus we have 
$$J_2(t)\leq e^{-\frac{t}{2\mu}} \frac{9}{10}\sum_{k=N_\mu^\lambda+1}^\infty \frac{1}{\mu\alpha_k}.$$
    To bound this summation, we need to make use of the asymptotic growth of $\alpha_k$. By the second assumption in \eqref{summoperator4.8},  
we have, for some universal constant $C$ and constant $C_{\eta_0}$ depending only on $\eta_0,$
\begin{equation}\label{whenmuislarge}
\begin{aligned}
\sum_{k=N_\mu^\lambda+1}^\infty \frac{1}{\mu\alpha_k} &\leq C\sum_{k=C^{-1}\mu^{-(1-\eta_0)}}^\infty \frac{1}{\mu k^{\frac{1}{1-\eta_0}}}\\&\leq C_{\eta_0} \mu^{-1}\int_{C^{-1}\mu^{-(1-\eta_0)}}^\infty x^{-\frac{1}{1-\eta_0}}dx\leq C_{\eta_0}\mu^{\eta_0-1}.\end{aligned}\end{equation}
Note also the elementary calculation
\begin{equation}\label{elementaryintegral}\int_0^\infty s^{-2\alpha}e^{-\frac{s}{2\mu}}ds=(2\mu)^{1-2\alpha}\int_0^\infty t^{-2\alpha}e^{-t}dt=C\mu^{1-2\alpha},\end{equation} 
so that
$$
\int_0^t (t-s)^{-2\alpha}J_2(t-s)ds\leq \int_0^t s^{-2\alpha}e^{-\frac{s}{2\mu}}\mu^{\eta_0-1}ds\leq \mu ^{\eta_0-2\alpha}.$$
For the term involving $J_1(t)$, we follow exactly the proof of Proposition \ref{proposition 1.8}. Combining everything, we now have
    $$\Lambda_\alpha^{\mu,\lambda}(t)\leq C\|\Phi \|_t^2 (\lambda^{-\eta}+\mu^{\eta_0-2\alpha})$$
    and then the Proposition follows from application of the BDG inequality.
\end{proof}

Before we proceed, we quote the following result on the $\mu\to 0$ convergence.
\begin{lemma}
    Let $f_\mu^\lambda(k)(t,u,v)$ be defined in \eqref{effkmu} with $\lambda>0$ fixed in the sequel. Then 
    \begin{itemize}
        \item For each $k\in\mathbb{N}$, $T>0$ and $u\in\mathbb{R}$,
        \begin{equation}
            \lim_{\mu\to 0}\sup_{t\in[0,T]}|f_\mu^\lambda(k)(t,u,o)-ue^{-(\alpha_k+\lambda)t}|=0.
        \end{equation}
        \item For each $k\in\mathbb{N},$ $T>0$, $t_0\in(0,T]$ and $v\in\mathbb{R}$,
        \begin{equation}
            \lim_{\mu\to 0}\sup_{t\in[t_0,T]}|f_\mu^\lambda(k)(t,0,v/\mu)-ve^{-(\alpha_k+\lambda) t}|=0.
        \end{equation}
        \item For each $k\in\mathbb{N}$, $t>0$ and $v\in\mathbb{R}$,
        \begin{equation} \label{uniformfantastic}       \lim_{\mu\to 0}f_\mu^\lambda(k)'(t,0,v)=0.
        \end{equation}
    \end{itemize}
\end{lemma}

\begin{proof}
    Let $f_\mu^\lambda(k)(t):=f_\mu^\lambda(k)(t,u,v)$, then it solves the ODE $\mu(f_\mu^\lambda(k))''+(f_\mu^\lambda(k))'+(\alpha_k+\lambda)f_\mu^\lambda=0$. The proof for the case $\lambda=0$ can be found in \cite{salins2019smoluchowski},Theorem 5.8, and the general case for $\lambda>0$ is exactly the same, so we omit the details. 
\end{proof}

Then we have to rework through the proof of Proposition \ref{proposition1.10!} to make sure the constants are $\mu$-independent for all $\mu\in(0,\mu_0(\lambda)]$. Recall the definition $N_\mu^\lambda=\max\{k:1-4\mu(\alpha_k+\lambda)\geq 0\}$, and consider $P_{N_\mu^\lambda}$ the projection onto the subspace spanned by $\{e_1,\cdots,e_{N_\mu^\lambda}\}$.

\begin{proposition}\label{uniformproposition1.10}
Fix $T>0$.   For each $p\geq 2$ and $\eta\in(0,\eta_0),$ we may choose $\alpha>0$ sufficiently small such that,
   for the damped stochastic wave equation $\zeta=1,$ 
\begin{enumerate}
    \item for $\mu\in(0,1)$ and $t\in[0,T],$
    \begin{equation}\label{uniform1.101.35}
\mathbb{E}|P_{N_\mu^\lambda}\Pi_2\Gamma_\alpha^{\mu,\lambda}(t)|_H^p\leq \frac{c_{\eta,p}}{\mu^p}\lambda^{-\frac{\eta p}{2}}\mathbb{E}\|\Phi\|_t^p,
    \end{equation}
    \item For any fixed $t\in[0,T]$ and any fixed $\lambda>0$,
    \begin{equation}\label{uniformgreats}
        \lim_{\mu\to 0}\mu^p \mathbb{E}|P_{N_\mu^\lambda}\Pi_2\Gamma_\alpha^{\mu,\lambda}(t)|_H^p=0.
    \end{equation}
    \item for any fixed $\lambda>0$, we may find a $\mu_0(\lambda)>0$ and a constant $\bar{\zeta}>0$ such that for all $t\in[0,T]$ and all $\mu\in(0,\mu_0(\lambda)]$ we have,
    \begin{equation}\label{prop4.33}
        \mathbb{E}|(I-P_{N_\mu^\lambda})\Pi_2\Gamma_\alpha^{\mu,\lambda}(t)|_{H^{-1}(\lambda)}^p\leq \frac{c_{\eta,p}}{\mu^{p(1-\bar{\zeta})/2}}\mathbb{E}\|\Phi\|_t^p.
    \end{equation}
\end{enumerate}
In all the above claims, the constant $C_{\eta,p}$ does not depend on the choice of $\lambda>0$ and the choice of $\mu\in(0,\mu_0(\lambda)]$.  

\end{proposition}
\begin{proof}
    Denote by $\Lambda_1(t)$ the quadratic variation of $P_{N_\mu^\lambda}\Pi_2 \Gamma_\alpha^{\mu,\lambda}$. We have
    $$\begin{aligned}
    \Lambda_1(t)&=\sum_{j=1}^\infty\int_0^t (t-s)^{-2\alpha}|P_{N_\mu^\lambda}\Pi_2\mathcal{S}_\mu^\lambda(t-s)\mathcal{I}_\mu\Phi(s)e_j|_H^2ds
    \\&=\sum_{k=1}^{N_\mu}\sum_{j=1}^\infty\int_0^t (t-s)^{-2\alpha}\langle \Phi(s)e_j,\mathcal{I}_\mu^*\mathcal{S}_\mu^{\lambda *}(t-s)\Pi_2^*e_k\rangle_H^2 ds \end{aligned} $$

Since $\mathcal{I}_\mu^*\mathcal{S}_\mu^{\lambda*}(t-s)\Pi_2^*e_k=f_\mu^\lambda(k)'(t-s,0,1/\mu)e_k$, we can bound 
\begin{equation}\Lambda_1(t)\leq\sum_{k=1}^{N_\mu^\lambda}\sum_{j=1}^\infty \int_0^t (t-s)^{-2\alpha}|(f_\mu^\lambda(k)'(t-s,0,1/\mu)|^2\langle\Phi(s)e_j,e_k\rangle_H^2 ds.
\end{equation}
Using the estimate
$$|f_\mu^\lambda(k)'(t,0,1/\mu)|\leq\frac{2e^{-(\alpha_k+\lambda) t}}{\mu},$$ we deduce that, following \eqref{J1J1} and the treatment of $J_1(t)$ in the proof of Proposition \ref{proposition 1.8},
$$\begin{aligned}\Lambda_1(t)&\leq \frac{C}{\mu^2}\sum_{k=1}^{N_\mu^\lambda}\int_0^t (t-s)^{-2\alpha}e^{-2(\alpha_k+\lambda)(t-s)}\|\Phi\|_s^2ds\\&\leq\frac{C}{\mu^2}\lambda^{-\eta}\|\Phi\|_t^2
.\end{aligned}$$ The proof of \eqref{uniform1.101.35} follows from application of BDG inequality:
$$\mathbb{E}|P_{N_\mu^\lambda}\Pi_1\Gamma_\alpha^{\mu,\lambda}(t)|_H^p\leq \mathbb{E}(\Lambda_1(t))^{p/2}.$$

For the second claim, since $\mu f_\mu^\lambda(k)'(t,0,1/\mu)=f_\mu^\lambda(k)'(t,0,1)$, using the estimate \eqref{uniformfantastic}, we have for each $s>0$ and $k\leq N_\mu^\lambda,$ $j\in\mathbb{N}$,
$$\lim_{\mu\to 0}(t-s)^{-2\alpha}\mu^2|f_\mu^\lambda(k)'(t,0,1/\mu)|^2\langle\Phi(s)e_j,e_k\rangle_H^2=0$$
and also the bound thanks to \eqref{zeta2est2}: as $k\leq N_\mu^\lambda$, 
$$\mu^2 (f_\mu^\lambda(k)'(t,0,1/\mu))^2\langle \Phi(s)e_j,e_k\rangle_H^2\leq Ce^{-2(\alpha_k+\lambda) t}\langle \Phi(s)e_j,e_k\rangle_H^2,$$
then by the dominated convergence theorem, $\Lambda_1(t)\to 0$ with probability one as $\mu\to 0$, and then \eqref{uniformgreats} follows from BDG inequality.

For the third claim, 
$$\begin{aligned}
\Lambda_2(t)=&\sum_{j=1}^\infty \int_0^t (t-s)^{-2\alpha}|(I-P_{N_\mu^\lambda})\Pi_2\mathcal{S}_\mu^\lambda(t-s)\mathcal{I}_\mu \Phi(s)e_j|_{H^{-1}(\lambda)}^2 ds\\
&=\sum_{j=1}^\infty\int_0^t (t-s)^{-2\alpha}|(A-\lambda)^{-1/2}(I-P_{N_\mu^\lambda})\Pi_2\mathcal{S}_\mu^\lambda(t-s)\mathcal{I}_\mu \Phi(s)e_j|_H^2ds\\
&\leq\sum_{k=N_\mu^\lambda+1}^\infty\sum_{j=1}^\infty \int_0^t (t-s)^{-2\alpha}\langle \Phi(s)e_j,\mathcal{I}_\mu^*\mathcal{S}_\mu^{\lambda*}(t-s)\Pi_2^*(I-P_{N_\mu^\lambda})^*(A-\lambda)^{-1/2}e_k\rangle_H^2 ds
\end{aligned}
$$
Since for $k,j\in\mathbb{N}$,
$$\begin{aligned}\langle \mathcal{I}_\mu^*\mathcal{S}_\mu^{\lambda*}(t-s)\Pi_2^*(I-P_{N_\mu^\lambda})^*(A-\lambda)^{-1/2}e_k,e_j\rangle_H\\
=\begin{cases}
    -(\alpha_k+\lambda)^{-1/2}f_\mu^\lambda(k)'(t-s,0,1/\mu)\quad \text{if } k=j>N_\mu^\lambda,\\0\quad \text{otherwise}.
\end{cases} \end{aligned}$$

By estimate \eqref{zeta4est4},
$$(\alpha_k+\lambda)^{-1/2}|f_\mu^\lambda(k)'(t-s,0,1/\mu)|\leq C(\alpha_k+\lambda)^{-1/2}\mu^{-1}e^{-\frac{t-s}{4\mu}},$$
we have, for all $\mu\in(0,\mu_0(\lambda)]$ with $\mu_0(\lambda)$ as defined in Proposition \ref{proposition 4.8extra},
$$\sum_{k=N_\mu^\lambda+1}^\infty(\alpha_k+\lambda)^{-1}|(f_\mu^\lambda(k)'(t-s,0,1/\mu))|^2
\leq\sum_{k=N_\mu^\lambda+1}^\infty \frac{Ce^{-\frac{t-s}{2\mu}}}{\mu^2(\alpha_k+\lambda)}\leq Ce^{-\frac{t-s}{2\mu}}\mu^{\eta_0-2},$$
where the last inequality is derived in exactly the same way as in \eqref{whenmuislarge}.
Thus we have the bound for $\Lambda_2(t)$:
$$\Lambda_2(t)\leq\frac{C}{\mu^{2-\eta_0}}\int_0^t (t-s)^{-2\alpha}e^{-\frac{t-s}{2\mu}}\|\Phi\|_s^2 ds.$$
Using \eqref{elementaryintegral}, we have 
$$\Lambda_2(t)\leq \frac{C}{\mu^{2\alpha+1-\eta_0}}\|\Phi\|_t^2 .$$
Now it suffices to take $\alpha>0$ small enough so that $2\alpha+1-\eta_0<1$. Then we can find $\bar{\zeta}>0$ such that  
$$\Lambda_2(t)\leq\frac{C}{\mu^{1-\bar{\zeta}}}\|\Phi\|_t^2.$$ The last claim then follows from BDG inequality.
\end{proof}

Now we adapt the proof of Theorem \ref{theorem2.11} to get $\mu$-independent estimates for small $\mu$.

\begin{theorem}\label{theorem2.11ext} Fix $T>0$.
Recall $\Gamma^{\mu,\lambda}$ defined in \eqref{Gammamulambda}. Then for any $p\geq 2$, any fixed $\lambda>0$ and any $\eta\in(0,\eta_0)$, we may find some $c_{\eta,p}>0$, some $\bar{\zeta}>0$ and some  $\mu_0(\lambda)>0$ such that for all $\mu\in(0,\mu_0(\lambda)],$ we have the estimate
\begin{equation}\label{1.38lasttheoremext}
\mathbb{E}\sup_{t\in[0,T]}|\Pi_1\Gamma^{\mu,\lambda}(t)|_H^p\leq c_{\eta,p}(\lambda^{-\frac{\eta p}{2}}+\mu^{\bar{\zeta}p/2})\mathbb{E}\|\Phi\|_T^p .
\end{equation}
The constant $c_{\eta,p}$ is independent of the choice of $\lambda>0$ and $\mu\in(0,\mu_0(\lambda)]$. We again note that for fixed $\lambda>0$, as long as $\mu_0(\lambda)$ is chosen small enough, the term $\mu^{\bar{\xi}p/2}$ can be discarded and the right hand side is indeed $\mu$-independent.

We also recall the following $\lambda=0$ version, which will be used in the subsequent proof:
\begin{equation}\label{salins}
    \mathbb{E}\sup_{t\in[0,T]}|\Pi_1\Gamma^\mu(t)|_H^p\leq C\mathbb{E}\int_0^T \sup_{s\in[0,t]}\|\Phi(s)\|_{\mathcal{L}(H_0,H)}^p dt,
\end{equation}
where $C$ is $\mu$-independent.
\end{theorem}

\begin{proof} Proof of claim \eqref{salins} can be found in \cite{salins2019smoluchowski}, Theorem 6.3. We will focus on the $\lambda>0$ case to get decay in $\lambda$:

    We apply the stochastic factorization as in \eqref{factorization} and treat each term separately.
For the first term
$$\mathbb{E}\sup_{t\in[0,T]}\left|\int_0^t (t-s)^{\alpha-1}\Pi_1\mathcal{S}_\mu^\lambda(t-s)\begin{pmatrix} I\\0\end{pmatrix} \Pi_1\Gamma_\alpha^{\mu,\lambda}(s)ds\right|_H^p$$
we use Lemma \ref{lemma 1.4} and Proposition \ref{proposition 4.8extra} instead of Proposition \ref{proposition 1.8} so that the numerical constant $c_{\eta,p}$ is $\mu$-independent. The difference compared with \eqref{firstexpanse} is that we have $\lambda^{-\frac{\eta p}{2}}+\mu^{(\eta_0-2\alpha)p/2}$ instead of $\lambda^{-\frac{\eta p}{2}}$ on the right hand side of the estimate.

For the second term
$$\mathbb{E}\sup_{t\in[0,T]}\left|\int_0^t (t-s)^{\alpha-1}\Pi_1\mathcal{S}_\mu^\lambda(t-s)\begin{pmatrix} 0\\P_{N_\mu^\lambda}\end{pmatrix} P_{N_\mu^\lambda}\Pi_2\Gamma_\alpha^{\mu,\lambda}(s)ds\right|_H^p
,$$
we use estimate \eqref{lemma2.51} and estimate \eqref{uniform1.101.35}. We have $\lambda^{-\frac{\eta p}{2}}$ on the right hand side and the numerical constant $c_{\eta,p}$ is $\mu$-independent.

For the third term
$$\mathbb{E}\sup_{t\in[0,T]}\left|\int_0^t (t-s)^{\alpha-1}\Pi_1\mathcal{S}_\mu^\lambda(t-s)\begin{pmatrix} 0\\I-P_{N_\mu^\lambda}\end{pmatrix} (I-P_{N_\mu^\lambda})\Pi_2\Gamma_\alpha^{\mu,\lambda}(s)ds\right|_H^p,$$
we use estimate \eqref{lemma2.52} and estimate \eqref{prop4.33}. We have $\mu^{\bar{\zeta}p/2}$ on the right hand side of the estimate and the numerical constant $c_{\eta,p}$ is $\mu$-independent.

Summing up the three estimates, the proof of the theorem is finished.
\end{proof}

\subsection{Small mass limit of stochastic integrals} We will need two more convergence results before we start the proof of small mass limit.

For the first result, recall the definition (for the stochastic heat equation)
\begin{equation}
    \Gamma(t)=\int_0^t \mathcal{S}(t-s)\Phi(s)dW_s
\end{equation}
where $\mathcal{S}$ is the analytic semigroup generated by $A$. Recall also the definition \eqref{uppergamma}
\begin{equation}
\Gamma^{\mu}(t)=\int_0^t\mathcal{S}_\mu(t-s)\Phi(s)dW_s.
\end{equation}
We need the following result on the convergence of $\Pi_1 \Gamma^{\mu}$ to $\Gamma$:
\begin{theorem}\label{theorem4.5}
    Given $T>0$ and $\alpha>0$, for any self-adjoint, progressively measurable $$\Phi\in L^p(\Omega;L^\infty([0,T];\mathcal{L}(H_0,H))),$$ any fixed $\lambda>0$ and any $p\geq 2$, let $\Gamma^{\mu}$ and $\Gamma$ be defined as before. Then 
    \begin{equation}\label{finalconv}
        \lim_{\mu\to 0}\mathbb{E}|\Pi_1\Gamma^{\mu}-\Gamma|^p_{\mathcal{C}([0,T];H)}=0.
    \end{equation}
\end{theorem}
This theorem is proved in \cite{salins2019smoluchowski}, Theorem 8.1. There was a $p\geq\frac{1}{\alpha}$ assumption in that result, but we can apply Cauchy-Schwartz to \eqref{finalconv} to reduce the value of $p$.

For Lebesgue integrals, we have
\begin{theorem}\label{theorem4.6} For any $T>0$ and $\varphi\in L^\infty([0,T];H),$
    \begin{equation}
        \lim_{\mu\to 0}\sup_{t\in[0,T]}\left|\int_0^t\left(\mathcal{S}(t-s)-\Pi_1\mathcal{S}_\mu(t-s)\mathcal{I}_\mu\right)\varphi(s)ds\right|_H=0.
    \end{equation}
\end{theorem}
The proof can be found in \cite{salins2019smoluchowski}, Theorem 8.2.

We also need the following more quantitative estimate in the exponent of $\lambda$ as $\lambda\to +\infty$.

\begin{theorem}\label{theorem4.7hhh}
    For any $T>0,$ $\lambda>0$ fixed there exists some $c_1(\lambda)>0$ such that for any $\mu\in(0,c_1(\lambda)]\cap(0,1),$ we have
\begin{equation}
    \sup_{t\in[0,T]}\left|\int_0^t \Pi_1\mathcal{S}_\mu^\lambda(t-s)\mathcal{I}_\mu\varphi(s)ds\right|_H\leq C\lambda^{-1}\sup_{t\in[0,T]}|\varphi(t)|_H,
\end{equation}
    where $C$ is some universal constant.
\end{theorem}
The same estimate is true if we replace $\Pi_1\mathcal{S}_\mu^\lambda(t-s)\mathcal{I}_\mu$ by $\mathcal{S}^\lambda$. Proof in the case of $\mathcal{S}^\lambda$ is very easy: note that $\|\mathcal{S}^\lambda(t-s)\|_H\leq e^{-\lambda(t-s)}$ and integrate over $[0,T]$.

\begin{proof}

    Recall from the proof of Lemma \ref{lemma2.3} that  $$\|\Pi_1\mathcal{S}_\mu^\lambda(t)\mathcal{I}_\mu\|_{\mathcal{L}(H)}\leq \sup_{k\in\mathbb{N}}|f_\mu^\lambda(k)(t,0,1/\mu)|.$$
    For $1\leq k\leq N_\mu^\lambda$, we have from \eqref{zeta1est1} that 
    $$|f_\mu^\lambda(k)(t,0,1/\mu)|\leq 4e^{-(\alpha_k+\lambda)t}.$$
    We choose $c_1(\lambda)=\frac{1}{5\lambda^2}$, so that for any $\mu\in(0,c_1(\lambda)]$ and $k\geq N_\mu^\lambda$, one must have $\alpha_k\geq \lambda^2$.
 
    Then for such $k$, by \eqref{zeta3est3} we have 
    $$|f_\mu^\lambda(k)(t,0,1/\mu)|\leq \frac{2}{\sqrt{(\mu(\alpha_k+\lambda)}}e^{-\frac{t}{4\mu}}\leq \frac{2}{\lambda}\frac{e^{-\frac{t}{4\mu}}}{\sqrt{\mu}}.$$

Then the claim of the theorem follows from noting that
$$\int_0^T 4e^{-(\alpha_k+\lambda)t}dt\leq C\lambda^{-1},\quad \int_0^T \frac{e^{-\frac{t}{4\mu}}}{\sqrt{\mu}}dt\leq 4\sqrt{\mu},$$
the fact that $\mu\in(0,1),$ and the fact that $\Pi_1\mathcal{S}_\mu^\lambda(t)\mathcal{I}_\mu$ preserves the one dimensional subspace spanned by each $e_k$.
\end{proof}

We also need the following result on convergence of initial values as $\mu\to 0$.

\begin{proposition}\label{ODEconverge}
Fix any initial value $(u_0,v_0)\in H\times H^{-1}$, and any $T>0$. Then we have

\begin{equation}
 \lim_{\mu\to 0}   \sup_{t\in[0,T]}|\mathcal{S}(t)u_0-\Pi_1\mathcal{S}_\mu(t)(u_0,v_0)|_H^2=0.
\end{equation}

\end{proposition}

This result is standard and only involves computations of ODEs. A detailed proof of this result can be found in \cite{salins2019smoluchowski}, Section 8.1, Proof of Theorem 4.2. Our assumptions on eigenvalues of $A$ are slightly more general but the proof is the same.

\subsection{Proof of Smoluchowski-Kramers approximation}
In this section we impose the slightly stronger semigroup assumption $\mathbf{H}1'$ and the $\alpha$- Hölder regularity assumption $\mathbf{H}2'$. 

For simplicity, we first assume the following uniform approximation assumption $\mathbf{H}7'$:

$\mathbf{H}7'$. There exists a sequence of maps $B^n(t,x):[0,\infty)\times H\to H$ and a sequence of maps $G^n(t,x):[0,\infty)\times H\to \mathcal{L}(H_0,H)$ such that they satisfy assumptions $\mathbf{H}2'$, $\mathbf{H}3,$ $\mathbf{H}4$ and $\mathbf{H}5$ with uniform in $n$ constants, and each $G^n$, $B^n$ are Lipschitz in $x$ with time-independent Lipschitz constants. Moreover, in the $n\to\infty $ limit, 
\begin{equation}\label{approx}
   g_n:= \sup_{t\geq 0}\sup_{x\in H} |G^n(t,x)-G(t,x)|_{\mathcal{L}(H_0,H)}\to 0,\quad  b_n:=\sup_{t\geq 0}\sup_{x\in H} |B^n(t,x)-B(t,x)|_H\to 0.
\end{equation}

As discussed in Section \ref{alternative}, Assumption $\mathbf{H}7'$ is already sufficiently general to cover many random field stochastic wave equations with irregular coefficients. Say if the drift coefficient $B(t,u)=b(t,u(t,x))$ for any $u\in L^2([0,1];\mathbb{R})$, and we can find a sequence of Lipschitz maps $b_n(t,x)$ converging uniformly to $b(t,x)$ on $\mathbb{R}$, then we have constructed an approximating sequence $B^n$ of $B$.  

The proof in the general case, i.e. without the assumption $\mathbf{H}7'$, only needs the following slight modifications: By \eqref{salins}, we obtain $\mu$-independent $L^p$ norm estimates for the process $u_\mu(t)$ on $\mathcal{C}([0,T];H)$ for any sufficiently small $\mu$. Now we can apply Proposition \ref{proposition}, combined with the compact embedding of $H^\delta$ into $H$ and
Markov's inequality (we should formulate a version of Proposition \ref{proposition} for the Lebesgue integration of the drift $B(t,u(t))$, which is easy to work out, and also note that $(u_0,v_0)\in\mathcal{H}_1$)
, we deduce that, for any $\epsilon>0$,  we can find a compact set $K_\epsilon\subset H$ such that with probability at least $1-\epsilon$ uniform in $\mu>0$, $u_\mu(t)\in K_\epsilon$ for all $t\in[0,T]$;  and with probability at least $1-\epsilon$, $u(t)\in K_\epsilon$ for all $t\in[0,T]$ . We set out to find Lipschitz approximations $B^n,G^n$ of $B$ and $G$ that converge uniformly on $K_\epsilon$, which is already outlined in Section \ref{767767}, and we run the approximating sequence with $B^n$ and $G^n$. Then the proof proceeds the same way as in Section \ref{weakunique}.

We consider the following four SPDEs: the first is the stochastic heat equation on $H$, with initial value $u_0\in H$. The weak well-posedness of \eqref{smalleqn1} was established in \cite{han2022exponential} (with time independent coefficients, but the proof is exactly the same for time dependent coefficients.)
\begin{equation}\label{smalleqn1}
\frac{\partial u(t)}{\partial t}=Au(t)+B(t,u(t))   +G(t,u(t)) \frac{dW_t}{dt}.
\end{equation}
The second is the stochastic heat equation with Lipschitz coefficients and the same initial condition
\begin{equation}\label{nsmalleqn1}
\frac{\partial u^n(t)}{\partial t}=Au^n(t)+B^n(t,u^n(t))   +G^n(t,u^n(t)) \frac{dW_t}{dt}.
\end{equation}
The third is the damped stochastic wave equation with initial value $(u_0,v_0)\in \mathcal{H}_1$, whose first coordinate we denote by $u_\mu(t).$
\begin{equation}\label{smalleqn2}
\mu \frac{\partial^2 u_\mu(t)}{\partial t^2}=Au_\mu(t) - \frac{\partial u_\mu(t)}{\partial t}+B(t,u_\mu(t)) +G(t,u_\mu(t)) \frac{dW_t}{dt}.
\end{equation}
The fourth is the damped stochastic wave equation with Lipschitz coefficients and the same initial value  $(u_0,v_0)$, whose first coordinate we denote by $u_\mu^n(t)$.
\begin{equation}\label{msmalleqn2}
\mu \frac{\partial^2 u_\mu^n(t)}{\partial t^2}=Au_\mu^n(t) - \frac{\partial u_\mu^n(t)}{\partial t}+B^n(t,u_\mu^n(t)) +G^n(t,u_\mu^n(t)) \frac{dW_t}{dt}.
\end{equation}

Recall the Wasserstein distance introduced in Section \ref{weakexistence}. We first prove the following proposition, inspired by \cite{kulik2020well}, Proposition 5.1:

\begin{proposition}\label{4.94.9}
    Denote by $c_n=\max(g_n,b_n)$ with $g_n,b_n$ defined in \eqref{approx}. Then we can find some constant $C>0$ and $\omega>0$ independent of $\mu$ such that
    \begin{equation}
        W(u^n_\mu|_{[0,T]},u_\mu|_{[0,T]})\leq C (c_n)^\omega.
    \end{equation}
\end{proposition}

\begin{proof}
    Consider an auxiliary SPDE 
    \begin{equation}\label{msmalleqn3}
\mu \frac{\partial^2 \widetilde{u}^n_\mu(t)}{\partial t^2}=A\widetilde{u}^n_\mu(t) - \frac{\partial \widetilde{u}^n_\mu(t)}{\partial t}+B^n(t,\widetilde{u}^n_\mu(t)) +\lambda(u_\mu(t)-\widetilde{u}^n_\mu(t))1_{t\leq\tau}
+G^n(t,\widetilde{u}^n_\mu(t)) \frac{dW_t}{dt},
\end{equation}
with the same initial value  $(u_0,v_0)$, whose first coordinate we denote by $\widetilde{u}_\mu^n(t)$.

We choose $\lambda=(c_n)^{\gamma-1}$ (in this proof we choose $\gamma$ and the forthcoming $\eta$ the same way as in Section \ref{weakunique}) and consider the stopping time $\tau:=\inf\{t\geq 0:|u_\mu(t)-\widetilde{u}_\mu^n(t)|\geq c_n\}$. Then via the same Girsanov transform argument, we have 
$$d_{TV}(u_\mu^n|_{[0,T]},\widetilde{u}_\mu^n|_{[0,T]})\leq C(c_n)^\gamma.$$

Now we do the pathwise estimate. Applying the $\mu$-independent estimate in Theorem \ref{theorem2.11ext}, we deduce that, in exactly the same way as in \eqref{869},   using also Theorem \ref{theorem4.7hhh} for the Lebesgue integral, the following $\mu$-independent estimate for all $
\mu>0$ sufficiently small: (more precisely, for all $\mu\in(0,\mu_0(\lambda)]$ for $\mu_0(\lambda)$ decreasing in $\lambda$) 

\begin{equation}
    \mathbb{P}(\sup_{t\in[0,T\wedge \tau]}|u_\mu(t)-\widetilde{u}_\mu^n(t)|\geq C (c_n)^{1+\chi_0})\leq c_n^\chi
\end{equation}
where $\chi>0$ and $\chi_0>0$ are some fixed constants. By sample path continuity and the fact that $\chi_0>0$, on the event $|u_\mu(t)-\widetilde{u}_\mu^n(t)|\leq C(c_n)^{1+\chi_0}$ for $t\in[0,T\wedge\tau]$ we must have $\tau>T$, hence
\begin{equation}
    \mathbb{P}(\sup_{t\in[0,T]}|u_\mu(t)-\widetilde{u}_\mu^n(t)|\geq C c_n)\leq c_n^\chi.
\end{equation}

By definition of total variation distance, we can find a coupling of $\widetilde{u}_\mu^n|_{[0,T]}$ and $u_\mu^n|_{[0,T]}$ such that they are identical on an event of probability at least $1-C(c_n)^\gamma$. Since we upper bound the metric by 1 in our definition of Wasserstein distance $W$, the expectation of transport mass over the complement set is at most $C(c_n)^\gamma$. This coupling of $u_\mu^n$ and $\widetilde{u}_\mu^n$ leads to a coupling of $u_\mu$ and $u_\mu^n$, where we simply couple $u_\mu$ and $\widetilde{u}_\mu^n$ via the same driving Wiener process. Combining all these estimates, and the fact that distance is upper bounded by 1, we deduce that for some $\omega>0$,
 \begin{equation}
        W(u^n_\mu|_{[0,T]},u_\mu|_{[0,T]})\leq C (c_n)^\omega.
    \end{equation} independent of the choice of $\mu>0$ sufficiently small.
\end{proof}

Exactly the same estimate applies to the stochastic heat equation \eqref{smalleqn1}, \eqref{nsmalleqn1}, yielding
 \begin{equation}\label{finalshes}
        W(u^n|_{[0,T]},u|_{[0,T]})\leq C (c_n)^\omega.
    \end{equation} 

We now give the proof of Theorem \ref{2smallnoise}, the small mass limit.

\begin{proof}

The target of our proof is to show that 
$$\lim_{\mu\to 0}W(u|_{[0,T]},u_\mu|_{[0,T]})=0.$$

By triangle inequality applied to the Wasserstein distance, we have
$$W(u|_{[0,T]},u_\mu|_{[0,T]})\leq W(u_\mu|_{[0,T]},u^n_\mu|_{[0,T]})+ W(u^n|_{[0,T]},u^n_\mu|_{[0,T]})+W(u^n|_{[0,T]},u|_{[0,T]}).$$

Then for each $n$ we have, applying Proposition \ref{4.94.9} and \eqref{finalshes}, the following estimate that does not depend on the value $\mu\in(0,\mu_0({c_n}^{\gamma-1}))$:
$$W(u|_{[0,T]},u_\mu|_{[0,T]})\leq W(u^n|_{[0,T]},u^n_\mu|_{[0,T]})+C(c_n)^\omega.$$

Thus we only have to show that, for each $n>0$, we have 
\begin{equation}
    \label{kuaqudaxue}
\lim_{\mu\to 0} W(u^n|_{[0,T]},u^n_\mu|_{[0,T]})=0.\end{equation}

Once this is done, the proof of the whole theorem is complete since $\lim_{n\to\infty} c_n=0$.

The claim \eqref{kuaqudaxue} is nothing but the Smoluchowski-Kramers approximation approximation of $u_\mu^n$ to $u^n$. Since the coefficients $G^n$ and $B^n$ are Lipschitz continuous in $x$, the said convergence follows from the existing literature \cite{cerrai2006smoluchowski}, \cite{cerrai2006smoluchowski2}, \cite{salins2019smoluchowski}. We give a sketch of proof for the claim \eqref{kuaqudaxue}, following the steps of \cite{salins2019smoluchowski}, for sake of completeness:

From the mild formulation \eqref{eq2.4geg} and \eqref{stochheat}, we have the following decomposition: 
\begin{equation}\label{monster}\begin{aligned}
     &u^n(t)-u_\mu^n(t)=(\mathcal{S}(t)u_0-\Pi_1 \mathcal{S}_\mu(t)(u_0,v_0))\\
  &+  \int_0^t (\mathcal{S}(t-s)-\Pi_1\mathcal{S}_\mu(t-s)\mathcal{I}_\mu) B^n(s,u^n(s))ds\\
&+\int_0^t \Pi_1 \mathcal{S}_\mu(t-s)\mathcal{I}_\mu(B^n(s,u^n(s))-B^n(s,u_\mu^n(s)))ds\\
&+[\int_0^t \mathcal{S}(t-s)G^n(s,u^n(s))dW(s)-\int_0^t \Pi_1 \mathcal{S}_\mu(t-s)\mathcal{I}_\mu G^n(s,u^n(s))dW(s)]\\
&+\int_0^t \Pi_1 \mathcal{S}_\mu(t-s)\mathcal{I}_\mu (G^n(s,u^n(s))-G^n(s,u_\mu^n(s))dW(s)\\
&=:\sum_{k=1}^5 J_k^\mu(t).
  \end{aligned}
\end{equation}

Applying Proposition \ref{ODEconverge}, Theorem \ref{theorem4.5} and \ref{theorem4.6}, we deduce that 
\begin{equation}
\lim_{\mu\to 0}\mathbb{E}\sup_{t\in[0,T]}|J_i^\mu(t)|_H^p=0,\quad i=1,2,4.
\end{equation}

Denote by $C_n$ the larger one of the Lipschitz constants of $G^n$ and $B^n$. Then, 
\begin{equation}
    \mathbb{E}\sup_{t\in[0,T]}|J_3(t)|_H^p\leq (C_n)^pT^{p-1}\mathbb{E}\int_0^T \sup_{s\in[0,t]}|u^n(s)-u_\mu^n(s)|_H^p dt.
\end{equation}
By estimate \eqref{salins} and the Lipschitz continuity of $G$,
\begin{equation}
    \mathbb{E}\sup_{t\in[0,T]}|J_5(t)|_H^p\leq (C_n)^p C(T)\mathbb{E}\int_0^T \sup_{s\in[0,t]}|u^n(s)-u_\mu^n(s)|_H^p dt.
\end{equation}

Taking the norm on both sides of \eqref{monster} and apply Gronwall's lemma, we deduce that 
$$\begin{aligned}\mathbb{E}&\sup_{t\in[0,T]}|u^n(t)-u_\mu^n(t)|_H^p\\&\leq C(n,T)e^{ TC(n,T)}(\sup_{t\in[0,T]}|J_1(t)|_H^p+\mathbb{E}\sup_{t\in[0,T]}|J_2(t)|_H^p+\mathbb{E}\sup_{t\in[0,T]}|J_4(t)|_H^p).
\end{aligned}$$
Taking $\mu\to 0$, we deduce that $\lim_{\mu\to 0}\mathbb{E}\sup_{t\in[0,T]}|u^n(t)-u_\mu^n(t)|_H=0$
and hence completes the proof of \eqref{kuaqudaxue}.
\end{proof}

\printbibliography

\end{document}